\documentclass[aap]{imsart}

\RequirePackage{amsthm,amsmath,amsfonts,amssymb}
\RequirePackage[numbers, sort]{natbib}
\RequirePackage[colorlinks,citecolor=blue,urlcolor=blue]{hyperref}%% uncomment this for coloring bibliography citations and linked URLs
\RequirePackage{graphicx}%% uncomment this for including figures
\usepackage{subfigure, hyperref}
\usepackage{comment}
\usepackage{xcolor}
\usepackage{dsfont}
\usepackage{bm}

\startlocaldefs
\numberwithin{equation}{section}
\theoremstyle{plain}
\newtheorem{theorem}{Theorem}[section]  
\newtheorem{corollary}[theorem]{Corollary}
\newtheorem{lemma}[theorem]{Lemma} 
\newtheorem{proposition}[theorem]{Proposition}

\theoremstyle{definition}
\newtheorem{remark}[theorem]{Remark} 
\newtheorem{example}[theorem]{Example}
\newtheorem{assumption}[theorem]{Assumption}
\newcommand{\rhs}{right-hand side}

\newcommand{\as}{{\rm a.s.}}
\newcommand{\cip}{\stackrel{\P}{\rightarrow}}
\newcommand{\cid}{\stackrel{\rm d}{\rightarrow}}
\newcommand{\cas}{\stackrel{\rm a.s.}{\rightarrow}}
\newcommand{\cw}{\stackrel{\rm w}{\rightarrow}}

\newcommand{\eid}{\stackrel{\rm d}{=}}

\newcommand{\x}{\mathbf{x}}

\newcommand{\y}{\mathbf{y}}
\newcommand{\X}{\mathbf{X}}

\newcommand{\R}{\mathbb{R}}  

\newcommand{\C}{\mathbb{C}}  
\newcommand{\N}{\mathbb{N}}

\newcommand{\B}{\mathbf{B}}
\newcommand{\A}{\mathbf{A}}
\newcommand{\I}{\mathbf{I}}
\newcommand{\Y}{{\mathbf Y}}
\newcommand{\bfT}{{\mathbf T}}

\renewcommand{\S}{\mathbf{S}}
\newcommand{\E}{\mathbb{E}}
\renewcommand{\P}{\mathbb{P}}

\newcommand\mf{\mathbf}
\newcommand\ml{\mathcal}
\newcommand\tp{\top}
\newcommand\wt{\widetilde}

\newcommand{\rw}{\rightarrow}
\newcommand{\twonorm}[1]{\|#1\|_2}
\newcommand{\sign}{\operatorname{sign}}
\newcommand{\Levy}{L\'{e}vy }

\renewcommand{\l}{\left}
\renewcommand{\r}{\right}

\newcommand{\be}{\begin{equation}}
	\newcommand{\ee}{\end{equation}}
\newcommand{\bea}{\begin{equation}\begin{aligned}}
		\newcommand{\eea}{\end{aligned}\end{equation}}
\newcommand{\tr}{\operatorname{tr}}

\newcommand{\Var}{\operatorname{Var}}

\renewcommand{\Im}{\operatorname{Im}}
\renewcommand{\Re}{\operatorname{Re}}
\newcommand{\vep}{\varepsilon}

\newcommand{\bfv}{{\bf v}}

\newcommand{\bfR}{{\bf R}}

\newcommand{\nto}{n\to\infty}
\newcommand{\e}{\mathbf{e}}
\newcommand{\MP}{Mar\v cenko--Pastur }

\newcommand{\bfx}{{\bf x}}

\newcommand{\bfZ}{{\bf Z}}
\newcommand{\bfz}{{\bf z}}
\newcommand{\bfS}{{\bf S}}

\newcommand{\bfI}{{\bf I}}

\newcommand{\norm}[1]{\|#1\|}
\newcommand{\diag}{{\rm diag}}

\newcommand{\1}{\mathds{1}}

\renewcommand{\geq}{\geqslant}
\renewcommand{\leq}{\leqslant}
\renewcommand{\ge}{\geqslant}
\renewcommand{\le}{\leqslant}

\newcommand{\zd}[1]{{#1}}
\newcommand{\jh}[1]{{#1}}

\endlocaldefs

\begin{document}

\begin{frontmatter}
\title{Quadratic form of heavy-tailed self-normalized random vector with applications in $\alpha$-heavy \MP law}
\runtitle{Quadratic form of heavy-tailed self-normalized random vector}

\begin{aug}
%%%%%%%%%%%%%%%%%%%%%%%%%%%%%%%%%%%%%%%%%%%%%%%
%% Only one address is permitted per author. %%
%% Only division, organization and e-mail is %%
%% included in the address.                  %%
%% Additional information such as            %%
%% identifying the corresponding author must %%
%% be included in in the Acknowledgments     %%
%% section if necessary.                     %%
\author[A]{\fnms{Zhaorui} \snm{Dong}
	\ead[label=e1,mark]{zhaoruidong@link.cuhk.edu.cn} }
\author[B]{\fnms{Johannes} \snm{Heiny}
	\ead[label=e2,mark]{heiny@kth.se} }
\author[A]{\fnms{Jianfeng} \snm{Yao}
	\ead[label=e3,mark]{jeffyao@cuhk.edu.cn} }
%%%%%%%%%%%%%%%%%%%%%%%%%%%%%%%%%%%%%%%%%%%%%%
%% Addresses                                %%
%%%%%%%%%%%%%%%%%%%%%%%%%%%%%%%%%%%%%%%%%%%%%%
\address[A]{The Chinese University of Hong Kong, Shenzhen,  \printead{e1,e3}}
\address[B]{KTH Royal Institute of Technology, \printead{e2}}
\end{aug}

\begin{abstract}
	Let $\mathbf{x}$ be a random vector with $n$ i.i.d.\ real-valued components in the domain attraction of an $\alpha$-stable law with  $\alpha\in(0,2)$, and let $\mathbf{y}=\mathbf{x}/\|\mathbf{x}\|_2$ be the associated self-normalized vector on the unit sphere. For a (possibly random) Hermitian matrix $\mathbf{A}_n=\big(a_{ij}^{(n)}\big)$ independent of $\mathbf{y}$, we study the asymptotic law of the quadratic form $\mathbf{y}^\top \mathbf{A}_n \mathbf{y}$. Building on the sharp separation between diagonal and off-diagonal contributions in this heavy-tailed setting, we show that under a mild assumption on the Frobenius  norm of the off-diagonal part of $\mathbf{A}_n$ the limiting law is solely governed by the empirical distribution of the diagonal entries and the index $\alpha$. More precisely, if $n^{-1}\sum_{i=1}^n \delta_{a^{(n)}_{ii}}$ converges weakly almost surely to a deterministic $\nu$, then $Q_n$ converges in distribution to a non-degenerate law $\mu_{\nu,\alpha}$  characterized through its Stieltjes transform. The law $\mu_{\nu,\alpha}$ is shown to be atom-free (provided that $\nu$ is non-degenerate) with an explicit density and tractable tail behavior.
	
	As an application in random matrix theory, we derive an implicit resolvent-based representation of the $\alpha$-heavy Mar\v{c}enko--Pastur law $H_{\alpha,\gamma}$ for heavy-tailed sample correlation matrices and prove that $H_{\alpha,\gamma}$ has no atoms except possibly at the origin.
	For comparison with the light-tailed setting, we also provide a Hanson--Wright-type concentration inequality for $\mathbf{y}^\top \mathbf{A}_n \mathbf{y}$ when the components of $\mathbf{x}$ are sub-Gaussian.
\end{abstract}



\end{frontmatter}
%%%%%%%%%%%%%%%%%%%%%%%%%%%%%%%%%%%%%%%%%%%%%%
%% Please use \tableofcontents for articles %%
%% with 50 pages and more                   %%
%%%%%%%%%%%%%%%%%%%%%%%%%%%%%%%%%%%%%%%%%%%%%%
%\tableofcontents

%%%%%%%%%%%%%%%%%%%%%%%%%%%%%%%%%%%%%%%%%%%%%%
%%%% Main text entry area:

\section{Introduction}

Consider an $n$-dimensional population vector  
\[
\bfx = (X_1, \ldots, X_n)^{\top} \in \R^n,
\]
where the components $X_i$ are independent and identically distributed (iid) copies of some non-degenerate random variable $\xi$ with $\E[\xi]=0$. For a symmetric matrix $\A\in \R^{n\times n}$, the quadratic form of $\x$ is defined as $Q_n(\x,\A)=\x^\top \A \x$.  Such quadratic forms of random vectors are  key objects in high-dimensional probability \cite{vershynin2009high}. When $\xi$ is sub-Gaussian, the sharp concentration of $Q_n(\x,\A)$ around its mean $\E[Q_n(\x,\A)]$ is given by the classical Hanson--Wright inequality \cite{hanson:wright:1971bound}.

\begin{theorem}[Hanson--Wright inequality for $Q_n(\x,\A)$, \cite{ hanson:wright:1971bound}]\label{thm:classical-HW}
	If $\xi$ is centered and sub-Gaussian, then there exist universal constants $C,c>0$ such that for any $n\in \mathbb{N}$ and $u>0$, 
	\begin{equation*}
		\P\!\left(\big|Q_n(\x,\A)-\E[Q_n(\x,\A)]\big|\ge u\right)\le
		C\exp\!\left[-c\,\min\!\left(\frac{u^2}{K^4\|\A\|_{\operatorname{F}}^2},\ \frac{u}{K^2\|\A\|}\right)\right],
	\end{equation*}
	where $K=\|\xi\|_{\psi_2}$.
\end{theorem}

Subsequent work has relaxed the tail conditions of $\xi$ or allowed dependence among the entries, see \cite{adamczak2015note, adamczak2012tail, dai2025note, louart2025} and the references therein. Among various dependence structures, the case that the random vector is distributed on the unit sphere $\ml{S}^{n-1}$ is of typical interest. A general way to produce a distribution on $\ml{S}^{n-1}$ is to consider the $L^2$ self-normalization of $\x$:
\begin{equation}\label{eq:L2-selfnormalization-y}
	\mf{y}=\l(\frac{X_1}{\sqrt{X_1^2+\cdots+X_n^2}},\cdots,\frac{X_n}{\sqrt{X_1^2+\cdots+X_n^2}}\r)^\top\in \ml{S}^{n-1}.
\end{equation}
In this work, we investigate the corresponding quadratic forms  $Q_n(\y,\A)=\y^\top \A\y$. A deterministic relation between $Q_n(\y,\A)$ and $Q_n(\x,\A)$ only exists in the special case where $\xi^2$ is non-random. For example, if $\xi$ has a Rademacher distribution, then one has $Q_n(\y,\A)=n^{-1} Q_n(\x,\A)$. 
\zd{If $\E[\xi^2]<\infty$, the denominator in \eqref{eq:L2-selfnormalization-y} can be replaced at first order by $\sqrt{n\E[\xi^2]}$, reducing the problem to the classical theory of the quadratic forms $Q_n(\x,\A)$. Even with such reduction, the limiting distribution of $Q_n(\x,\A)$ and possible characterizations thereof can significantly vary depending on $\A$. It is proved in \cite{bentkus1999optimal} that $Q_n(\x,\A)$ converges in distribution under spectral conditions of $\A$. Gaussian and Gamma limiting distributions were studied in \cite{nualart2005central,nourdin2009noncentral,nourdin2010invariance,nourdin2015optimal}. More recently, in the case that $\A$ is a $\{0,1\}$-valued symmetric matrix with zero diagonals, Bhattacharya et al.\ \cite{bhattacharya2024fluctuations} characterized all possible limits of $Q_n(\x,\A)$ as combinations of Gaussian, weighted centered chi-square, and Gaussian variance-mixture components. Thus, light tails or boundedness of $\xi$ alone do not determine the limit of $Q_n(\x,\A)$ when the matrix $\A$ varies with $n$. When $\E[\xi^2] = \infty$, the limit distribution of $Q_n(\x,\A)$ may be heavy-tailed; see  \cite{bai:ginovyan:taqqu:2016, bai:taqqu:2016} for quadratic forms of \Levy\-driven linear processes. In comparison to the amount of literature in the light-tailed case, 
limit theory for $Q_n(\x,\A)$ and $Q_n(\y,\A)$ remains limited under $\E[\xi^2] = \infty$.}

\medskip

The asymptotic behaviors of the quadratic forms $Q_n(\x,\A)$ and $Q_n(\y,\A)$ find applications in many fields. In this paper, we will focus on a close connection with random matrix theory (RMT), especially with the study of the limiting eigenvalue distribution of sample covariance and correlation matrices. 

To this end, for i.i.d.\ random variables $\{X_{ij}: i,j\ge 1\}$ with $X_{11}\eid \xi$, we construct the data matrix  
\[
\X = \X_n = (X_{ij})_{1 \le i \le p, \, 1 \le j \le n},
\]
the sample covariance matrix $\bfS = \bfS_n = n^{-1} \X \X^{\top}$, and the sample correlation matrix
\begin{align}\label{eq:defRY_new}
	\bfR = \bfR_n = \{\diag(\bfS_n)\}^{-1/2} \, \bfS_n \, \{\diag(\bfS_n)\}^{-1/2} = \Y \Y^{\top}.
\end{align}
Here, the standardized matrix $\Y = \Y_n = (Y_{ij})_{1 \le i \le p,\, 1 \le j \le n}$ has entries
\begin{equation}\label{eq:Yij_new}
	Y_{ij} = Y_{ij}^{(n)} = \frac{X_{ij}}{\sqrt{X_{i1}^2 + \cdots + X_{in}^2}}, \qquad 1 \le i \le p,\, 1 \le j \le n,
\end{equation}
which depend on $n$.  
Since $Y_{ij}$ is invariant under scaling of the $X_{ij}$’s, we will without loss of generality assume $\E[\xi^2] = 1$ whenever $\E[\xi^2]$ is finite.

We will often assume that $\xi$ has a regularly varying tail with index $\alpha > 0$, namely
\begin{equation}\label{eq:regvar}
	\P(|\xi| > x) = x^{-\alpha} L(x), \qquad x > 0,
\end{equation}
where $L$ is slowly varying at infinity\footnote{A function $L:(0,\infty) \to (0,\infty)$ is \emph{slowly varying} if $L(tx)/L(x) \to 1$ as $x \to \infty$ for all $t>0$.}.  
Regularly varying distributions have power-law tails, and all moments of $|\xi|$ of order strictly greater than $\alpha$ are infinite.  
Typical examples include the Pareto distribution with parameter $\alpha$ and the $t$-distribution with $\alpha$ degrees of freedom.  

Our main interest is in the \emph{high-dimensional regime}
\[
p = p_n \to \infty, \qquad \text{as} \quad n \to \infty,
\]
and in particular the \emph{proportional growth regime} $p/n \to \gamma \in (0, \infty)$. In the proportional regime, the spectral properties of the sample covariance matrix $\bfS$ have been extensively studied in RMT since the seminal work of \cite{marchenko:pastur:1967}, which shows that the empirical distribution of the eigenvalues $\lambda_i(\bfS)$ converges weakly to the celebrated Mar\v{c}enko--Pastur (MP) law.  
Moreover, \cite{BS04} established asymptotic normality of suitably centered and normalized linear spectral statistics of $\bfS$ in this regime when $\E[\xi^4] < \infty$.

% In contrast, the study of the sample correlation matrix $\bfR$ is more recent and more limited.  
% A key difficulty is that, compared to the original data matrix $\X$, the entries of the standardized matrix $\Y$ in \eqref{eq:Yij_new} are no longer independent within the same row (although the different rows remain iid).  
% This dependence makes $\bfR$ more challenging to analyze.  
% Under the proportional regime, Lemma~2 in \cite{bai:yin:1993} shows that $\E[\xi^4] < \infty$ is equivalent to
% \[
% \|\diag(\bfS) - \bfI_p\| \cas 0, \quad n \to \infty,
% \]
% where $\|\cdot\|$ is the spectral norm and $\bfI_p$ is the $p$-dimensional identity matrix; see also \cite[Theorem 1.2]{heiny:2022} for a related result in the dependent case.  
% Consequently, under a finite fourth moment, the normalization $\{\diag(\bfS)\}^{-1/2}$ in \eqref{eq:defRY_new} can be replaced by $\bfI_p$, and therefore
% \[
% \max_i |\lambda_i(\bfR) - \lambda_i(\bfS)| \;\le\; \|\bfR - \bfS\| \;\to\; 0 \quad \text{almost surely}.
% \]

The asymptotic behavior of quadratic forms $Q_n(\x,\A)$ and $Q_n(\y,\A)$ closely connects with spectral properties of $\S$ and $\mf{R}$, respectively. Yaskov \cite{yaskov:2016} proved that for sample covariance matrix, the LSD converges to MP law if and only if  $n^{-1}(Q_n(\x,\A)-\E[Q_n(\x,\A)])$ converges to zero in probability. The papers \cite{dong:yao:2025} and \cite[Theorem 2.1]{doernemann:heiny:2025} extend this equivalence result to sample correlation matrices by deriving the necessary and sufficient condition $Q_n(\y,\A)-\E[Q_n(\y,\A)]\cip 0$. 
Furthermore, if $\xi$ is regularly varying with index $\alpha\in (0,2)$, then the LSD is the $\alpha$-heavy \MP~law $H_{\alpha,\gamma}$ \cite{Heiny:Yao:AOS}.
This suggests that the quadratic form $Q_n(\y,\A)$ exhibits different behaviors than in the light-tailed case. 

In \cite{Heiny:Yao:AOS}, it is proved that $H_{\alpha,\gamma}$ converges to a zero-inflated Poisson distribution as $\alpha\downarrow 0$,  which has countably many atoms in the LSD; as $\alpha\uparrow 2$, $H_{\alpha,\gamma}$ converges to the classical MP law, which has no atom except at 0. It remains mysterious whether or not atoms exist in the LSD for the general case $\alpha\in (0,2)$: is the countably many atoms a transition phenomena at $\alpha=0$ or at some $\alpha_0\in (0,2)$? The detection of such atoms is unfeasible by the moment method which is used in the proof of \cite{Heiny:Yao:AOS}.

\medskip\noindent{\bf Contributions of this paper.} 
In this work, we derive the limiting distribution of $Q_n(\y,\A)$ when $\xi$ is regularly varying with index $0<\alpha<2$. We firstly establish the limiting distribution for the case that diagonal elements of $\A$ are uniformly bounded, and then we extend to the unbounded case. We also discuss the tail probability of $Q_n(\y,\A)$. The limiting law of $Q_n(\y,\A)$ allows an implicit representation of the $\alpha$-heavy MP law $H_{\alpha,\gamma}$. From this representation, we prove that $H_{\alpha,\gamma}$ has no atom except at zero. The typical difficulty in the proof is that the resolvent's diagonal entries do not concentrate anymore around their mean and thus,  local laws as in the light-tailed case cease to hold. To settle this difficulty, we represent the Stieltjes transform of the $\alpha$-heavy MP law with the weak limit point of the resolvent's diagonal entries. The no-atom result is then deduced by analyzing the boundary behavior of this  representation. We also provide a matrix construction for the boundary case $\alpha=0$ in the $\alpha$-heavy MP law.

For independent interest, in the appendix we give a Hanson--Wright inequality for $Q_n(\y,\A)$ when $\xi$ is sub-Gaussian. 

\medskip\noindent{\bf Notations.}\qquad
Convergence in distribution (resp.\ probability) is denoted by $\cid$ (resp.\ $\cip$), and equality in distribution by $\eid$. Weak convergence of measure is denoted by $\cw$. 
Unless explicitly stated otherwise, all limits are for $n \to \infty$.  
A probability measure $\mu$ is \emph{degenerate} if $\mu = \delta_b$ for some $b \in \R$, where $\delta_b$ is the Dirac measure at $b$.  
We denote $\lambda$ as the Lebesgue measure on $\R$, and $\N_0=\N\cup\{0\}$.

For an $n \times n$ Hermitian matrix $\A$ we denote its eigenvalues by $\lambda_1(\A) \ge \cdots \ge \lambda_n(\A)$ and the   
empirical spectral distribution (ESD) of $\A$ is given by 
$\mu_\A = \frac{1}{n} \sum_{i=1}^n \delta_{\lambda_i(\A)}$.
A probability measure $\mu$ on $\R$ is uniquely characterized by its {\em Stieltjes transform}
\begin{equation*}
	s_{\mu}(z)= \int_{\R} \frac{1}{x-z} \, \mu(dx)\,, \quad z\in \C\backslash \operatorname{supp}(\mu)\,.
\end{equation*} 
Weak convergence of a sequence of probability measures $(\mu_n)$ to $\mu$ is equivalent to $s_{\mu_n}(z) \to s_{\mu}(z)$ a.s.~for all $z\in \C^+$ (see \cite[Chapter~3]{bai:silverstein:2010} or \cite{yao:zheng:bai:2015}).
The operator (resp. Frobenius) norm of a matrix $\A$ is denoted by $\|A\|$ (resp. $\|A\|_F)$. The Euclidean norm of a vector $\x$ is $\twonorm{\x}=\sqrt{\x^\top \bar{\x}}$.

For sequences $(a_n)$ and $(b_n)$ we write
$a_n = O(b_n)$ if $|a_n/b_n| \le C$ for some $C>0$ and all sufficiently large $n$,
$a_n = o(b_n)$ if $a_n/b_n \to 0$,
$a_n \sim b_n$ if $a_n/b_n \to 1$,
and $a_n \lesssim b_n$ if $a_n \le C b_n$ for some $C>0$ and all sufficiently large $n$.

We write $\mathcal{C}_b(\mathbb{R})$ for the space of bounded continuous functions on $\mathbb{R}$.
We use the principal branch of the complex logarithm (and power) with branch cut along $(-\infty,0]$.
Let $\Gamma(\cdot)$ and $B(\cdot,\cdot)$ denote the gamma and beta functions.

\section{Preliminaries} \label{sec:preliminaries}

We study quadratic forms constructed from self-normalized random vectors such as $\y$ in \eqref{eq:L2-selfnormalization-y} or the columns $\y_1, \ldots, \y_p$ of the matrix $\Y^{\top}$ in \eqref{eq:defRY_new}. We recall that the random vectors $\y, \y_1, \ldots, \y_p$ have the same distribution and therefore, we choose to formulate our results for 
$$\y_1=(Y_{11},\ldots, Y_{1n})^\top\,.$$
Unless explicitly stated otherwise, we assume that $\xi$ is symmetric and regularly varying with index $\alpha \in (0,2)$, that is, $\xi \eid -\xi$ and the regular variation condition \eqref{eq:regvar} holds. This means that $\xi$ is heavy-tailed with infinite variance.
From its definition in \eqref{eq:defRY_new} it follows that $Y_{ij}$ also possesses a symmetric distribution. 

For integers $k_1,\ldots,k_r\ge 1$, we need the notation 
$$
\beta_{2k_1,2k_2,\cdots,2k_r}=\E[Y_{11}^{2k_1}Y_{12}^{2k_2}\cdots Y_{1r}^{2k_r}]
$$
for the mixed moments of the components of $\y_1$. Their asymptotic behavior is described in the following lemma, whose proof can be found in \cite[p.~4]{albrecher:teugels:2007}.
\begin{lemma}\label{lem:allmoments}
	Assume that $\xi$ is regularly varying with index $\alpha \in (0,2)$. Then it holds
	\begin{equation}\label{moment}
		\lim_{\nto} n^r \beta_{2k_1,\ldots, {2k_r}} = \frac{\Big(\frac{\alpha}{2}\Big)^{r-1} \Gamma(r) \prod_{j=1}^r \Gamma(k_j-\alpha/2)}{ \Big(\Gamma(1-\alpha/2)\Big)^r \Gamma(k_1+\cdots+k_r)}\,.
	\end{equation}
	In particular, we have
	\begin{equation}\nonumber
		\lim_{\nto} n \beta_{2k} = \frac{\Gamma(k-\alpha/2)}{\Gamma(1-\alpha/2) \Gamma(k)}\,, \qquad k\ge 1\,.
	\end{equation}
\end{lemma}

\begin{remark}
	It is interesting to compare the order of those mixed moments with the case of more light-tailed distributions of $\xi$. For example, if $\E[\xi^{2\max_i k_i}]<\infty$, then one can derive similarly to \cite{albrecher:teugels:2007} that
	\begin{equation*}
		\lim_{\nto} n^{k_1+\cdots+k_r} \beta_{2k_1,\ldots, {2k_r}} = \prod_{i=1}^r \E[\xi^{2 k_i}]\,.
	\end{equation*}
	A glance at \eqref{moment} shows that $\beta_{2k_1,\ldots, {2k_r}}$ is of order $n^{-r}$ in the heavy-tailed case, while it is only of order $n^{-(k_1+\cdots+k_r)}$ in the light-tailed case. This stark difference ultimately results in $\y_1^\top \A\y_1$ concentrating around its mean in the light-tailed case, whereas  concentration does not happen in the infinite variance setting.
\end{remark}

Now let $\{\A_n\}_{n=1}^\infty$ be a sequence of $n\times n$ random Hermitian matrices independent of $\y_1$. We write $\A_n=(a_{ij}^{(n)})_{i,j=1}^n$ to highlight that the entries $a_{ij}^{(n)}$ may depend on $n$. In our proofs, we will typically omit superscript $(n)$ for simplicity.
Our main aim is to provide limit theory for the quadratic forms
\begin{equation*}
	Q_n=Q_n(\y_1,\A_n)=\y_1^\top \A_n\y_1\,.
\end{equation*}
$Q_n$ can be decomposed into its diagonal and off-diagonal parts as follows:
\bea\label{Def-Qn1Qn2}
Q_{n,1}=\mf{y}_1^\top \operatorname{\diag}(\A_n) \mf{y}_1,\qquad Q_{n,2}=\mf{y}_1^\top \big(\A_n-\operatorname{\diag}(\A_n)\big) \mf{y}_1.
\eea

\jh{
\begin{remark}
For slowly varying distributions of $\xi$, that is $\P(|\xi|>x)=L(x)$ with slowly varying function $L$, it is well-known (e.g.~\cite[Theorem~1]{haeusler:mason:1991}) that 
\begin{equation*} 
\E\big[ |Y_{1k_1}|\big] \to 1\,, \qquad \nto\,,
\end{equation*}
where $k_1=\arg \max_{j=1,\ldots, n} |Y_{1j}|$.
This in agreement with the principle of ``one big jump'' which is folklore in extreme value theory. This principle says that one of the heavy-tailed $|X_{1j}|$'s is exceptionally large compared to the rest. The heavier the tails, the more pronounced this effect becomes.
Lemma~\ref{lem:structure} gives the following approximation for $\y_1$,
\begin{equation}\label{eq:y1}
\twonorm{\y_1-\sign(Y_{1k_1}) \e_{k_1}}\cip 0\,,
\end{equation}
where $\e_i$ is the $i$-th canonical basis vector of $\R^n$. We see that
\begin{equation*}
Q_n(\sign(Y_{1k_1}) \e_{k_1},\A_n)=Q_{n,1}(\sign(Y_{1k_1}) \e_{k_1},\A_n)
\end{equation*}
since the off-diagonal part $Q_{n,2}(\sign(Y_{1k_1}) \e_{k_1},\A_n)=\e_{k_1}^\top \big(\A_n-\operatorname{\diag}(\A_n)\big) \e_{k_1} =0$ vanishes.
\end{remark}

For $\alpha\in(0,2)$, the approximation in \eqref{eq:y1} does not hold anymore. Nevertheless, the heuristic about the off-diagonal part remains asymptotically valid under a mild condition on $\A_n$.}

\zd{\begin{assumption}[Off-diagonal negligibility]\label{ass:offdiag}
	The off-diagonal part of $\A_n$ satisfies
	\begin{equation}\label{eq:offdiag-assumption}
\lim_{\nto} n^{-2} \E[\|\A_n-\operatorname{\diag}(\A_n) \|_{\operatorname{F}}^2] =0.
	\end{equation}
\end{assumption}}
\begin{lemma}\label{Prop-quadratic-form-Qn2}
	\zd{Under Assumption~\ref{ass:offdiag} it holds $Q_{n,2}\cip 0$ as $\nto$.}
\end{lemma}
\begin{proof}
	Since the $Y_{1i}$'s are also symmetrically distributed, we get
	\begin{align*}
		\E[|Q_{n,2}|^2]&=\sum_{1\le i_1\ne i_2 \le n}\sum_{1\le j_1\ne j_2\le n}\E\big[a_{i_1i_2}\overline{a_{j_1j_2}}\big]\E[Y_{1i_1}Y_{1i_2}Y_{1j_1}Y_{1j_2}]\\
		&=2\sum_{1\le i_1\ne i_2\le n} \E\big[|a_{i_1 i_2}|^2\big]\E[Y_{11}^2 Y_{12}^2]
		= 2 \,\E[\|\A_n-\operatorname{\diag}(\A_n) \|_{\operatorname{F}}^2] \, \beta_{2,2}\\ 
		&\le \frac{2 \,
			\E[\|\A_n-\operatorname{\diag}(\A_n) \|_{\operatorname{F}}^2]}{n^2}\to 0\,,
	\end{align*}
	where $n^2\beta_{2,2}\le 1$ was used for the inequality in the last line.
	Then Markov's inequality implies $Q_{n,2}\stackrel{\P}{\rw}0$.
\end{proof}

Lemma~\ref{Prop-quadratic-form-Qn2} shows that the asymptotic behavior of $Q_n=Q_{n,1}+Q_{n,2}$ is determined by $Q_{n,1}$ provided that $\E[\|\A_n-\operatorname{\diag}(\A_n) \|_{\operatorname{F}}^2]=o(n^2)$. Assumption~\ref{ass:offdiag} is satisfied by several natural classes of non-diagonal matrices, as given in the following example. In particular, the case that $\A$ is a resolvent matrix satisfies Assumption \ref{ass:offdiag}, which plays an important role in the application to RMT.
\medskip

\zd{\begin{example}
For instance, if $ \E\|\A_n\|^2=o(n)$, then
\[
	\E\|\A_n-\operatorname{\diag}(\A_n)\|_{\operatorname{F}}^2
	\leq \E\|\A_n\|_{\operatorname{F}}^2
	\leq n\,\E\|\A_n\|^2=o(n^2),
\]
and hence Assumption~\ref{ass:offdiag} holds. Therefore, it is always satisfied for resolvent matrices $\A_n$ since in this case $\sup_{n\ge 1} \E\|\A_n\|^2<\infty$. Resolvents are an essential tool for the analysis of spectral properties of large random matrices \cite{yao:zheng:bai:2015, bai:silverstein:2010}. Furthermore, \eqref{eq:offdiag-assumption} holds for orthogonal projection matrices. Indeed, if $\mathbf{P}_n$ is an $n$-dimensional projection matrix, then
\[
	\|\mathbf{P}_n-\operatorname{\diag}(\mathbf{P}_n)\|_{\operatorname{F}}^2
	=\operatorname{rank}(\mathbf{P}_n)-\sum_{i=1}^n (\mathbf{P}_n)_{ii}^2
	\leq n.
\]
In particular, for a Haar orthogonal projection $\mathbf{P}_n=\mathbf{U}_n\operatorname{\diag}(\I_{r_n},0)\mathbf{U}_n^\top$ with $r_n/n\to\gamma\in[0,1]$, its diagonal empirical distribution converges weakly almost surely to the Dirac measure $\delta_\gamma$. (Indeed, $(\mathbf{P}_n)_{ii}\sim \operatorname{Beta}(r_n/2,(n-r_n)/2)$, and standard beta concentration, a union bound and the Borel--Cantelli lemma yield $\max_i|(\mathbf{P}_n)_{ii}-r_n/n|\to0$ almost surely.) \\
Assumption~\ref{ass:offdiag} also holds for uniformly bounded sparse matrices with $o(n^2)$ non-zero off-diagonal entries, including band matrices with bandwidth $w_n=o(n)$, bounded-degree graph adjacency matrices, block diagonal matrices with maximal block size $o(n)$, and Toeplitz-type matrices with square-summable off-diagonal profiles.
\end{example}}
\medskip

\zd{On the other hand, if Assumption~\ref{ass:offdiag} fails, the limiting distribution of $Q_n$ can vary a lot. This is illustrated in the two following examples and in Section~\ref{sec:additional_examples}. 
\begin{example}
\label{ex:resolvent}
For $\theta\ne 0$, take $\mf{A}_n=(1-\theta)\mf{I}_n+\theta \mf{1}_n\mf{1}_n^\top$, so that $a_{ii}^{(n)}=1$ and $a_{ij}^{(n)}=\theta$ for $i\ne j$. Then
$$
n^{-2}\|\mf{A}_n-\diag(\mf{A}_n)\|_{\operatorname{F}}^2\rw \theta^2>0.
$$
Thus the Frobenius-norm condition \eqref{eq:offdiag-assumption} fails. On the other hand, since $\sum_{i=1}^n Y_{1i}^2=1$ it is easy to see that
$$
Q_n=1-2\theta+\theta\l(\sum_{i=1}^n Y_{1i}\r)^2.
$$
By the the theory for self-normalized sums \cite{logan:mallows:rice:shepp:1973}, the limiting distribution of $\sum_{i=1}^n Y_{1i}$ is a stable law with infinite variance. Consequently, varying $\theta$ yields a family of possible weak limits for $Q_n$. It is worth pointing out that in this example the random contribution to $Q_n$ comes from $Q_{n,2}$, whose limiting distribution cannot be characterized through moments.
\end{example}
}

\medskip

\zd{Below, we give another example which illustrates that failure of Assumption~\ref{ass:offdiag} can still lead to a limiting distribution of $Q_n$ with all moments finite. 
\begin{example}
Let $\A_n$ be independent of $\y_1$, with $a_{ii}^{(n)}=0$ and 
$a_{ij}^{(n)}=a_{ji}^{(n)}=G_{ij}$, $1\le i<j\le n$,
where the $G_{ij}$'s are iid $\ml{N}(0,1)$. Then
\[
	n^{-2}\E\|\A_n-\diag(\A_n)\|_{\operatorname{F}}^2
	=\frac{n(n-1)}{n^2}\rw 1,
\]
so Assumption \ref{ass:offdiag} fails. On the other hand,
\[
	Q_n=Q_{n,2}=2\sum_{1\le i<j\le n}G_{ij}Y_{1i}Y_{1j}.
\]
Conditionally on $\y_1$, $Q_n$ is a linear combination
of independent Gaussian variables $G_{ij}$ with variance 
\[
	4\sum_{1\le i<j\le n}Y_{1i}^2Y_{1j}^2
	=2\left\{\left(\sum_{i=1}^nY_{1i}^2\right)^2-\sum_{i=1}^nY_{1i}^4\right\}
	=2(1-R_n)\,,
\]
where $R_n=\sum_{i=1}^n Y_{1i}^4$. Therefore, we obtain
\[
	Q_n \eid \sqrt{2(1-R_n)}\,Z\,,
\]
where $Z\sim\ml{N}(0,1)$ is independent of $\y_1$.  Since $X_{1i}^2$ is regularly varying with index $\alpha/2$, the usual point process convergence for regularly varying random variables \cite[Proposition~4.2]{albrecher:teugels:2007} yields that $R_n \cid R_{\alpha}:=U/V^2$, where the dependent random variables $U$ and $V$ follow $\alpha/4$ (resp.) $\alpha/2$-stable distributions.
%\cite[Corollary 1]{LePage1981ConvergenceStableDistribution} yields that $R_n$ converges to some distribution $R_\alpha$.
Consequently,
\[
	Q_n\cid Q_\alpha:=\sqrt{2(1-R_\alpha)}\,Z, \qquad \nto\,,
\]
where $Z\sim\ml{N}(0,1)$ is independent of $R_\alpha$. 
%The limit is non-degenerate: by Lemma~\ref{lem:allmoments},
%\[\E R_\alpha=\lim_{\nto} n\beta_4=1-\frac{\alpha}{2}, \qquad\text{and hence}\qquad \E Q_\alpha^2=\alpha>0.\]
%Moreover, for every $\ell\in\N$,
%\[\lim_{\nto}\E[Q_n^{2\ell}]=
	%\frac{(2\ell)!}{\ell!}\,\E\big[(1-R_\alpha)^\ell\big]\leq \frac{(2\ell)!}{\ell!},\qquad\lim_{\nto}\E[Q_n^{2\ell-1}]=0.
%\]
%In fact, $\E e^{tQ_\alpha}=\E e^{t^2(1-R_\alpha)}\le e^{t^2}$ for all $t\in\R$, so the limiting law is characterized by its moments. 
Using the fact that $R_\alpha \in [0,1]$, it is easy to verify Carleman's condition for $Q_\alpha$.
Hence, this example shows that the off-diagonal part can contribute to a non-degenerate limit distribution which is uniquely characterized through its moments.
\end{example}}
\medskip

\zd{The preceding examples show why some control of the off-diagonal part, like Assumption~\ref{ass:offdiag}, is
unavoidable if the limit is to be characterized explicitly; see also the additional examples in Section~\ref{sec:additional_examples}. The limiting distribution of $Q_n$ is highly sensitive to the structure of $\A_n$. Such sensitivity is consistent with the literature on the light-tailed case.}

\section{Main results}
In this section, we consider the setting that was introduced at the beginning of Section~\ref{sec:preliminaries}. In particular, we assume that $\xi$ is symmetric and regularly varying with index $\alpha \in (0,2)$. The main objective will be to describe the limiting distribution of $Q_{n,1}$ defined in \eqref{Def-Qn1Qn2}.

\subsection{Limiting law: bounded diagonal case}

First, we derive the limiting law of $Q_{n,1}$ in the case where the diagonal entries of $\A_n$ are uniformly bounded.
\begin{theorem}\label{Prop-quadratic-form}
	If $M:=\limsup_{n\rw\infty}\max_{1\leq i\leq n}|a_{ii}^{(n)}|<\infty$ almost surely and $\frac{1}{n}\sum_{i=1}^n\delta_{a_{ii}^{(n)}}$ converges weakly almost surely to a deterministic probability measure $\nu$, then $Q_{n,1}$ converges in distribution to a random variable with induced probability measure $\mu_{\nu,\alpha}$, whose Stieltjes transform is given by
	\bea\label{Eq-Stieltjes-transform-Quadratic-form}
	s_{\mu_{\nu,\alpha}}(z)=-\frac{\int_{\R}(z-x)^{\frac{\alpha}{2}-1}\nu(dx)}{\int_{\R}(z-x)^{\frac{\alpha}{2}}\nu(dx)},\qquad z\in\C\backslash\operatorname{supp}(\nu).
	\eea
	Moreover, $\mu_{\nu,\alpha}$ is uniquely characterized by its moments $m_\ell=\int_\R x^\ell \mu_{\nu,\alpha}(dx)$, $\ell\ge 1$, where
	\begin{equation*}
		m_\ell=\sum_{r=1}^\ell \l(\frac{\alpha}{2}\r)^{r-1} \frac{\ell}{r\, (\Gamma(1-\frac{\alpha}{2}))^r}
		\sum_{\substack{\ell_1,\cdots,\ell_r\geq 1\\ \ell_1+\cdots+\ell_r=\ell}} \prod_{s=1}^r \frac{\Gamma\l(\ell_s-\frac{\alpha}{2}\r)}{\ell_s!} \int_\R x^{\ell_s}\nu(dx)\,,   
	\end{equation*}
	and the support of $\mu_{\nu,\alpha}$ is contained in $[-M,M]$.
\end{theorem}
\begin{proof}
	We start with the case that $\{\A_n\}_{n=1}^\infty$ are deterministic.  By the multinomial theorem we have for $\ell\in \N$,
	\begin{align*}
		\E[Q_{n,1}^\ell]&= \E\Big[\big(a_{11} Y_{11}^2+ a_{22}Y_{12}^2 + \cdots+ a_{nn} Y_{1n}^2 \big)^\ell \Big]\\
		&=\sum_{\substack{\ell_1,\ldots,\ell_n\geq 0\\\ell_1+\cdots+\ell_n=\ell}}\frac{\ell!}{\ell_1!\cdots \ell_n!}a_{11}^{\ell_1}\cdots a_{nn}^{\ell_n}\E[Y_{11}^{2\ell_1}\cdots Y_{1n}^{2\ell_n}]\\
		&=\sum_{r=1}^\ell\sum_{\substack{\ell_1,\ldots,\ell_r\geq 1\\\ell_1+\cdots+\ell_r=\ell}}\frac{\ell!}{\ell_1!\cdots \ell_r!\, {r!}}  \l( \frac{1}{n^r}  \sum_{1\le i_1\ne\cdots\ne i_r\le n}a_{i_1i_1}^{\ell_1}\cdots a_{i_ri_r}^{\ell_r}\r) n^r\beta_{2\ell_1,\cdots,2\ell_r}.
	\end{align*}
	In view of $M:=\limsup_{n\rw\infty}\max_{1\leq i\leq n}|a_{ii}|<\infty$, we deduce that
	\begin{align*}
		\frac{1}{n^r}\sum_{i_1\ne\cdots\ne i_r}a_{i_1i_1}^{\ell_1}\cdots a_{i_ri_r}^{\ell_r}=\frac{1}{n^r}\prod_{s=1}^r{\sum_{i=1}^n a_{ii}^{\ell_s}}+O\l(\frac{1}{n}\r), \qquad \nto\,.
	\end{align*}
	Now using $n^{-1}\sum_{i=1}^n\delta_{a_{ii}}\cw\nu$ and the boundedness of $|a_{ii}|$'s, we obtain
	\begin{align*}
		\lim_{n\rw\infty}\frac{1}{n^r}\prod_{s=1}^r{\sum_{i=1}^n a_{ii}^{\ell_s}}=\prod_{s=1}^r\int_\R x^{\ell_s}\nu(dx).
	\end{align*}
	In conjunction with \eqref{moment}, we conclude that $\lim_{n\rw\infty}\E[Q_{n,1}^\ell]=m_{\ell}$, where
	\begin{align}\label{eq:dgseree}
		m_\ell&=\sum_{r=1}^\ell\sum_{\substack{\ell_1,\cdots,\ell_r\geq 1\\\ell_1+\cdots+\ell_r=\ell}}\frac{\ell!}{\ell_1!\cdots \ell_r!r!}\l(\prod_{s=1}^r\int_\R x^{\ell_s}\nu(dx)\r)\l(\frac{\alpha}{2}\r)^{r-1}\frac{\Gamma(r)\prod_{s=1}^r\Gamma(\ell_s-\frac{\alpha}{2})}{(\Gamma(1-\frac{\alpha}{2}))^r\Gamma(\ell)}.
	\end{align}
	By Fatou's lemma,
	\bea\nonumber
	|m_\ell|\leq \limsup_{n\rw\infty} \E[|Q_{n,1}|^\ell]\leq \E[\limsup_{n\rw\infty}|Q_{n,1}|^\ell]\leq \E\l[(M\sum_{i=1}^nY_{1i}^2)^\ell\r]=M^\ell.
	\eea
	Hence, by Carleman's condition, the moment sequence $\{m_\ell\}_{\ell=1}^\infty$ uniquely determines a probability measure $\mu_{\nu, \alpha}$, and $Q_{n,1}\cid\mu_{\nu, \alpha}$, as $\nto$. We set $$D_1=\{\zeta\in \C\backslash\operatorname{supp}(\nu):\, |\zeta|>M\} \quad \text{ and } \quad
	D_2=\{\zeta\in \C\backslash\operatorname{supp}(\nu):\, |\zeta|<M^{-1}\},
	$$
	and we define
	$$
	k_{r,\ell}:=\sum_{\substack{\ell_1,\ldots,\ell_r\geq 1\\ \ell_1+\cdots+\ell_r=\ell}}\prod_{s=1}^r\frac{\Gamma(\ell_s-\frac{\alpha}{2})\int_\R x^{\ell_s}\nu(dx)}{\Gamma(1-\frac{\alpha}{2})(\ell_s!)}.
	$$
	Observe that $k_{r,\ell}$ is the coefficient of $z^\ell$ in the formal power series $\left[K_x(z)\right]^r$, where
	$$
	K_x(z):=\sum_{j=1}^\infty \frac{\Gamma\left(j-\frac{\alpha}{2}\right)\int_\R x^{j}\nu(dx)}{\Gamma\left(1-\frac{\alpha}{2}\right)(j!)}z^j\,, \qquad z\in D_2\,.
	$$
	From \eqref{eq:dgseree} and the definition of $k_{r,\ell}$ we get that 
	\begin{align*}
		\frac{m_\ell }{\ell} &=\sum_{r=1}^{\ell} \frac{1}{r}\l(\frac{\alpha}{2}\r)^{r-1}\sum_{\substack{\ell_1,\ldots,\ell_r\geq 1\\ \ell_1+\cdots+\ell_r=\ell}}\prod_{s=1}^r\frac{\Gamma(\ell_s-\frac{\alpha}{2})\int_\R x^{\ell_s}\nu(dx)}{\Gamma(1-\frac{\alpha}{2})(\ell_s!)}
		= \sum_{r=1}^{\ell} \frac{1}{r}\l(\frac{\alpha}{2}\r)^{r-1} k_{r,\ell}
	\end{align*}
	is the coefficient of $z^\ell$ in the formal power series $H(z)$, where
	$$
	H_x(z):=\sum_{r=1}^\infty \frac{1}{r}\left(\frac{\alpha}{2}\right)^{r-1}\left[K_x(z)\right]^r=-\frac{2}{\alpha}\log\left(1-\frac{\alpha}{2}K_x(z)\right)\,.
	$$
	The last step follows from the identity $- \log(1-z)= \sum_{n=1}^\infty z^n/n$, which is valid for $|z|\le 1, z\neq 1$. Using the identity
	$$(1-xz)^{\alpha/2}=\sum_{j=0}^\infty (-xz)^j \prod_{k=1}^j \frac{\alpha/2 -k+1}{k}=1-\frac{\alpha}{2} \sum_{j=1}^\infty (xz)^j \frac{\Gamma\left(j-\frac{\alpha}{2}\right)}{\Gamma\left(1-\frac{\alpha}{2}\right)(j!)}$$
	and the fact that $\nu$ is a probability measure, we conclude that
	\begin{align*}
		\sum_{\ell=1}^\infty \frac{m_\ell z^\ell}{\ell} &=H_x(z)
		=-\frac{2}{\alpha}\log\left(1-\frac{\alpha}{2}K_x(z)\right)\\
		&=-\frac{2}{\alpha}\log\l(\int_\R \l(1-xz\r)^{\frac{\alpha}{2}}\nu(dx)\r)\,,\qquad z\in D_2\,.
	\end{align*}
	Taking derivatives with respect to $z$ on both sides, we obtain
	\bea\nonumber
	\sum_{\ell=1}^\infty m_\ell \,z^{\ell-1}=\frac{\int_\R \l(1-xz\r)^{\frac{\alpha}{2}-1}x\nu(dx)}{\int_\R \l(1-xz\r)^{\frac{\alpha}{2}}\nu(dx)},\qquad z\in D_2,
	\eea
	and 
	\bea\nonumber
	\sum_{\ell=1}^\infty m_{\ell} \,z^{-\ell+1}=\frac{\int_\R \l(1-\frac{x}{z}\r)^{\frac{\alpha}{2}-1}x\nu(dx)}{\int_\R \l(1-\frac{x}{z}\r)^{\frac{\alpha}{2}}\nu(dx)},\qquad z\in D_1.
	\eea
	Using the relationship $s_{\mu_{\nu,\alpha}}(z)=-z^{-1}-\sum_{\ell=1}^\infty m_\ell \,z^{-\ell -1}$ between the Stieltjes transform and the moments of $\mu_{\nu,\alpha}$, we get
	\begin{align*}
		s_{\mu_{\nu,\alpha}}(z)&=-\frac{1}{z}-\frac{1}{z^2}\frac{\int_\R \l(1-\frac{x}{z}\r)^{\frac{\alpha}{2}-1}x\nu(dx)}{\int_\R \l(1-\frac{x}{z}\r)^{\frac{\alpha}{2}}\nu(dx)}\\
		&=-\frac{\int_{\R}(z-x)^{\frac{\alpha}{2}-1}\nu(dx)}{\int_{\R}(z-x)^{\frac{\alpha}{2}}\nu(dx)},\quad z\in D_1.
	\end{align*}
	Since $s_{\mu_{\nu,\alpha}}(z)$ is analytic in $\C\backslash\operatorname{supp}(\nu)$, the proof in the case where $\{\A_n\}_{n=1}^\infty$ are deterministic.
	
	Finally, if $\A_n$ is random, it is independent of $\y_1$ by assumption. Therefore, the proof works analogously by conditioning on $\{\A_n\}_{n=1}^\infty$.
\end{proof}

\begin{remark} (a) The denominator of \eqref{Eq-Stieltjes-transform-Quadratic-form} is non-zero. To see this, take $z\in \C^+$ for instance. Then one has for $x\in \R$ that $\arg(z-x)\in(0,\pi)$. Since $\alpha\in(0,2)$,
	\bea\nonumber
	\Im \l((z-x)^{\frac{\alpha}{2}}\r)=|z-x|^{\frac{\alpha}{2}}\sin\l(\frac{\alpha}{2}\arg(z-x)\r)>0
	\eea
	and therefore, $\Im\int_{\R}(z-x)^{\frac{\alpha}{2}}\nu(dx)>0$. The other cases are similar.\\
	(b) The assumption that the probability measure $v$ is deterministic is crucial. To this end, consider the setting where $\A_n=Z \, \bfI_n$ with a non-degenerate random variable $Z$. In this case, we have $Q_{n,1}=\y_1^\top \diag(\A_n) \y_1=Z$ and $\frac{1}{n}\sum_{i=1}^n\delta_{a_{ii}^{(n)}}=\delta_Z$.
\end{remark}
A combination of Theorem~\ref{Prop-quadratic-form} and Lemma~\ref{Prop-quadratic-form-Qn2} immediately yields the following result.
\begin{corollary}\label{corr:mainthm}
	Under the conditions of Theorem~\ref{Prop-quadratic-form} and  Assumption~\ref{ass:offdiag}, then, as $\nto$, $Q_{n}$ converges in distribution to a random variable with induced probability measure $\mu_{\nu,\alpha}$, whose Stieltjes transform is given by \eqref{Eq-Stieltjes-transform-Quadratic-form}.
\end{corollary}

Next, we study the weak limits of $\mu_{\nu,\alpha}$ as $\alpha$ tends to two or zero, respectively. 
\begin{proposition}\label{prop:munualpha_examples}
	Consider the probability measure $\mu_{\nu,\alpha}$ with Stieltjes transform $s_{\nu,\alpha}$ in \eqref{Eq-Stieltjes-transform-Quadratic-form}, where $\nu$ is a probability measure on $\R$ and $\alpha\in (0,2)$. Then the following statements hold.
	\begin{itemize}
		\item[$(i)$] As $\alpha\uparrow 2$, $\mu_{\nu,\alpha}\cw\delta_{\int x\nu(dx)}$.
		\item[$(ii)$] As $\alpha\downarrow 0$, $\mu_{\nu,\alpha}\cw\nu$.
	\end{itemize}
\end{proposition}
\begin{proof}
	For $z\in\C\backslash\operatorname{supp}(\nu)$, we have
	\bea\nonumber
	s_{\mu_{\nu,\alpha}}(z)\rw-\frac{1}{\int_\R(z-x)\nu(dx)}=\frac{1}{\int_\R x\nu(dx)-z}=s_{\delta_{\int_\R x\nu(dx)}}(z)\,, \qquad \alpha\uparrow2\,,
	\eea
	and
	\bea\nonumber
	s_{\mu_{\nu,\alpha}}(z)\rw-\int_{\R}\frac{\nu(dx)}{z-x}=s_\nu(z)\,, \qquad \alpha\downarrow 0\,.
	\eea   
\end{proof}
It is worth mentioning that the boundary cases $\alpha=2$ and $\alpha=0$ play special roles and have nice interpretations. First, note that $\alpha=2$ constitutes the boundary of finite and infinite $\E[\xi^2]$. More precisely, for $\alpha>2$ the second moment $\E[\xi^2]$ is finite, whereas it is infinity for $\alpha\in(0,2)$. At $\alpha=2$, $\E[\xi^2]$ can be finite or infinite depending on the slowly varying function $L$ in\eqref{eq:regvar}.

Proposition~\ref{prop:munualpha_examples} $(i)$ shows that as $\alpha$ approaches $2$, the limit of $\mu_{\nu,\alpha}$ becomes degenerate and $\nu$ only influences the limit through its mean. This would suggest that  
\begin{equation}\label{eq:efsesdgfdsf}
	Q_{n,1} \cip  \int_\R x\nu(dx)\,, \qquad \nto\,.
\end{equation}
It turns out that this heuristic argument can be rigorously formulated as follows. 
\begin{proposition}\label{lem:1jh}
	If $M:=\limsup_{n\rw\infty}\max_{1\leq i\leq n}|a_{ii}^{(n)}|<\infty$ almost surely, $\frac{1}{n}\sum_{i=1}^n\delta_{a_{ii}^{(n)}}$ converges weakly almost surely to a deterministic probability measure $\nu$ and $\xi$ is the domain of attraction of the normal distribution, then \eqref{eq:efsesdgfdsf} holds.
\end{proposition}
\begin{proof}
	Since $\E[Q_{n,1}|\A_n]=\frac{1}{n} \tr(\A_n) \to \int_\R x\nu(dx)$ by assumption, it follows that $\E[Q_{n,1}|\A_n]-\E[Q_{n,1}]\cip 0$, as $\nto$. Thus, by Slutsky's theorem, it suffices to prove
	\begin{equation}\label{eq:dgfd}
		Q_{n,1}-\E[Q_{n,1}|\A_n] \cip 0\,, \qquad \nto\,.
	\end{equation}
	We know from \cite[Theorem~2.4]{doernemann:heiny:2025} that $n\beta_4\to 0$ is equivalent to $\xi$ being in the domain of attraction of the normal distribution.
	Using $M<\infty$ and $n\beta_4 +n(n-1)\beta_{2,2}=1$, we get
	\begin{align*}
		\Var\big( Q_{n,1}-\E[Q_{n,1}|\A_n] \big) &=\sum_{i=1}^n \E[a_{ii}^2] (\beta_4-n^{-2}) + \sum_{1\le i\neq j\le n} \E[a_{ii}a_{jj}] (\beta_{2,2}-n^{-2})\\
		&\lesssim M^2 (n\beta_4 +n^{-1}) + M^2 n^{-1}\to 0,\qquad \nto\,,
	\end{align*}
	establishing the desired result \eqref{eq:dgfd} in view of Markov's inequality.
\end{proof}
\begin{remark}
	We remark that if $\xi$ satisfies the regular variation condition \eqref{eq:regvar} for some $\alpha>0$ and is in the domain of attraction of normal distribution, then $\alpha$ must necessarily be at least two. Thus, Proposition~\ref{lem:1jh} nicely complements Theorem~\ref{Prop-quadratic-form}. For sub-Gaussian $\xi$, Theorem~\ref{thm:HW-Qy} provides a stronger concentration result than \eqref{eq:efsesdgfdsf} for $Q_n$, which might be of independent interest.   
\end{remark}
While Proposition~\ref{prop:munualpha_examples} $(i)$ and Proposition \ref{lem:1jh} reveal that $Q_{n,1}$ becomes degenerate if $\xi$ is light-tailed, Proposition~\ref{prop:munualpha_examples} $(ii)$ sheds light on the extremely heavy-tailed case $\alpha \downarrow 0$. 
The weak limits in Proposition~\ref{prop:munualpha_examples} $(i)$ and $(ii)$ coincide if and only if $\nu$ is degenerate, that is, $\nu=\delta_b$ for some $b\in \R$. In this case, we have for all $\alpha\in (0,2)$ that $\mu_{\delta_b,\alpha}=\delta_b$ since
\bea\nonumber
s_{\mu_{\delta_b,\alpha}}(z)=-\frac{(z-b)^{\frac{\alpha}{2}-1}}{(z-b)^{\frac{\alpha}{2}}}=\frac{1}{b-z}=s_{\delta_b}(z).
\eea
In particular, if $\A_n=b \, \bfI_n$, then $Q_{n,1}=\y_1^\top \diag(\A_n) \y_1=b$ holds for any $n\in \N$.
\medskip

Our next goal is to derive the density of $\mu_{\nu,\alpha}$ for non-degenerate $\nu$. 
To this end, we introduce the non-tangential limit as in Section 3.1 of \cite{mingo2017free}. Let $a\in\R$ and suppose $f:\C^+\rw\C$. We say $\lim _{\sphericalangle z \rightarrow a} f(z)=b$ if for every $\theta>0$, $\lim_{z\rw a}f(z)=b$ when we restrict $z$ to be in the cone
$$
\{u+iv:\ v>0,\ |u-a|<\theta v\}\subset \C^+.
$$

\begin{theorem}\label{thm::quadratic-from-property}
	Under the assumptions of Theorem \ref{Prop-quadratic-form} and if $\nu$ is non-degenerate, the probabilty measure $\mu_{\nu,\alpha}$ with Stieltjes transform \eqref{Eq-Stieltjes-transform-Quadratic-form} has no atom and is absolutely continuous with a density function given by
	{\footnotesize\bea\label{Eq-Density-of-Quadratic-Form}
		f_{\nu,\alpha}(x)=\frac{\sin\l(\frac{\alpha\pi}{2}\r)}{\pi}\frac{\int_\R(x-y)_+^{\frac{\alpha}{2}}\nu(dy)\cdot\int_\R(x-y)_-^{\frac{\alpha}{2}-1}\nu(dy)+\int_\R(x-y)_-^{\frac{\alpha}{2}}\nu(dy)\cdot\int_\R(x-y)_+^{\frac{\alpha}{2}-1}\nu(dy)}{(\int_\R(x-y)_+^{\frac{\alpha}{2}}\nu(dy))^2+(\int_\R(x-y)_-^{\frac{\alpha}{2}}\nu(dy))^2+2\cos\l(\frac{\alpha\pi}{2}\r)\int_\R(x-y)_+^{\frac{\alpha}{2}}\nu(dy)\int_\R(x-y)_-^{\frac{\alpha}{2}}\nu(dy)}.
		\eea}
	The density $f_{\nu,\alpha}$ is finite $\lambda$-a.e.\ on $\R$, strictly positive in $K_\nu= (\inf\operatorname{supp}(\nu),\sup\operatorname{supp}(\nu))$, and is zero outside $K_\nu$.
\end{theorem}
\begin{proof}
	From Theorem \ref{Prop-quadratic-form} and Lemma \ref{Lemma-Recover-atom}, we obtain for any $a\in\R$,
	\begin{align}
		\mu_{\nu,\alpha}(\{a\})=\lim_{\sphericalangle z\rw a}(a-z)s_{\mu_{\nu,\alpha}}(z) 
		&=\lim_{\sphericalangle z\rw a}\frac{(z-a)\int_\R(z-x)^{\frac{\alpha}{2}-1}\nu(dx)}{\int_\R(z-x)^{\frac{\alpha}{2}}\nu(dx)} \notag \\
		&=\lim_{\sphericalangle z\rw a}\frac{(z-a)^{\frac{\alpha}{2}}\nu(\{a\})}{\int_\R(z-x)^{\frac{\alpha}{2}}\nu(dx)}. \label{eq:ldef}
	\end{align}
	By definition of the non-tangential limit, for any $a,x\in\R,\ a\ne x$,
	\bea\label{Eq-spherical-limit-of-power}
	\lim_{\sphericalangle z\rw a} (z-x)^{\frac{\alpha}{2}}&=(a-x)_+^{\frac{\alpha}{2}}+e^{\frac{i\alpha\pi}{2}}(a-x)_-^{\frac{\alpha}{2}},\\
	\lim_{\sphericalangle z\rw a} (z-x)^{\frac{\alpha}{2}-1}&=(a-x)_+^{\frac{\alpha}{2}-1}-e^{\frac{i\alpha\pi}{2}}(a-x)_-^{\frac{\alpha}{2}-1},
	\eea
	and $\lim_{\sphericalangle z\rw a} (z-a)^{\frac{\alpha}{2}}=0$.
	Let $\sigma=\nu-\nu(\{a\})\delta_a$, then an application of the dominated convergence theorem gives
	\begin{align*}
		\xi_a:=\lim_{\sphericalangle z\rw a}\int_\R(z-x)^{\frac{\alpha}{2}}\nu(dx)&= \lim_{\sphericalangle z\rw a} \nu(\{a\})(z-a)^{\frac{\alpha}{2}}+\lim_{\sphericalangle z\rw a}\int_\R(z-x)^{\frac{\alpha}{2}}\sigma(dx)\,,\\
		&=\int_\R(a-x)_+^{\frac{\alpha}{2}}\sigma(dx)+e^{\frac{i\alpha\pi}{2}}\int_\R(a-x)_-^{\frac{\alpha}{2}}\sigma(dx)
	\end{align*}
	and 
	\begin{align*}
		\Re(\xi_a) &= \int_\R(a-x)_+^{\frac{\alpha}{2}}\sigma(dx)+\cos\l(\frac{\alpha\pi}{2}\r)\int_\R(a-x)_-^{\frac{\alpha}{2}}\sigma(dx)\\
		\Im (\xi_a) &= \sin\l(\frac{\alpha\pi}{2}\r)\int_\R(a-x)_-^{\frac{\alpha}{2}}\sigma(dx).
	\end{align*}
	Since $\nu$ is non-degenerate, $\sigma$ is non-zero and $\sigma(\{a\})=0$. Then at least one of the integrals
	$$
	\int_\R(a-x)_+^{\frac{\alpha}{2}}\sigma(dx) \quad \text{ and } \quad \int_\R(a-x)_-^{\frac{\alpha}{2}}\sigma(dx)
	$$
	is positive, and as a consequence at least one of $\Re(\xi_a),\ \Im(\xi_a)$ is non-zero. We deduce $\xi_a\ne 0$ and using \eqref{eq:ldef}
	\bea\nonumber
	\mu_{\nu,\alpha}(\{a\})=\frac{\nu(\{a\})}{\xi_a}\lim_{\sphericalangle z\rw a} (z-a)^{\frac{\alpha}{2}}=0.
	\eea
	Hence $\mu_{\nu,\alpha}$ has no atom.
	
	Now we turn to the proof of the density formula for $\mu_{\nu,\alpha}$. By virtue of Lemma~\ref{Lemma-a.e.finite-alphaPotential} with $\beta:=1-\alpha/2$, we have
	$$I_{\beta,\nu}(x):=\int_\R \frac{\nu(dy)}{|x-y|^\beta}<\infty$$
	for all $x\in H$, where $H$ is $\R$ without some set of Lebesgue measure zero. Now for $x\in H$, let $z=x+iv,\ v\in(0,\varepsilon)$. In view of
	\begin{align*}
		&|z-y|^{\frac{\alpha}{2}-1}\leq |x-y|^{\frac{\alpha}{2}-1},\qquad  \int_\R |x-y|^{\frac{\alpha}{2}-1}\nu(dy)=I_{\beta_0,\nu}(x)<\infty,\\
		&|z-y|^{\frac{\alpha}{2}}\leq(|x-y|+\varepsilon)^{\frac{\alpha}{2}},\qquad  \int_\R(|x-y|+\varepsilon)^{\frac{\alpha}{2}}\nu(dy)<\infty,
	\end{align*}
	the dominated convergence theorem and \eqref{Eq-spherical-limit-of-power} imply
	\bea\nonumber
	\lim_{v\downarrow 0}\int_\R (z-y)^{\frac{\alpha}{2}-1}\nu(dy)=\int_\R(x-y)_+^{\frac{\alpha}{2}-1}\nu(dy)-e^{\frac{i\alpha\pi}{2}}\int_\R(x-y)_-^{\frac{\alpha}{2}-1}\nu(dy),
	\eea
	\bea\nonumber
	\lim_{v\downarrow 0}\int_\R (z-y)^{\frac{\alpha}{2}}\nu(dy)=\int_\R(x-y)_+^{\frac{\alpha}{2}}\nu(dy)+e^{\frac{i\alpha\pi}{2}}\int_\R(x-y)_-^{\frac{\alpha}{2}}\nu(dy).
	\eea
	Since $\mu_{\nu,\alpha}$ has no atom at $x$, we can apply Theorem B.8 in \cite{bai:silverstein:2010} and Theorem \ref{Prop-quadratic-form}, to obtain
	\bea\nonumber
	f_{\nu,\alpha}(x)&=\frac{1}{\pi}\lim_{v\downarrow 0}\Im s_{\mu_{\nu,\alpha}}(z)\\
	&=-\frac{1}{\pi}\Im\l(\frac{\int_\R(x-y)_+^{\frac{\alpha}{2}-1}\nu(dy)-e^{\frac{i\alpha\pi}{2}}\int_\R(x-y)_-^{\frac{\alpha}{2}-1}\nu(dy)}{\int_\R(x-y)_+^{\frac{\alpha}{2}}\nu(dy)+e^{\frac{i\alpha\pi}{2}}\int_\R(x-y)_-^{\frac{\alpha}{2}}\nu(dy)}\r).
	\eea
	Direct calculation gives \eqref{Eq-Density-of-Quadratic-Form}. 
	
	Now we prove that $f_{\nu,\alpha}$ integrates to 1, so that $\mu_{\nu,\alpha}$ has no singular components and hence is absolutely continuous. Let $\theta=\alpha\pi/2\in (0,\pi)$. For $x\in\R$, define
	$$
	A(x)=\int_\R (x-y)_+^\frac{\alpha}{2}\nu(dy),\quad B(x)=\int_\R (x-y)_-^{\frac{\alpha}{2}}\nu(dy),\quad F(x)=A(x)+e^{i\theta}B(x).
	$$
	Then $A(x)\geq 0,\ B(x)\geq 0,\ \Im F(x)\geq 0$. Since $\nu$ is non-degenerate, $A(x)+B(x)>0$, hence $F(x)\ne 0$. For almost every $x\in\R$,
	$$
	A'(x)=\frac{\alpha}{2}\int_\R(x-y)_+^{\frac{\alpha}{2}-1}\nu(dy),\quad B'(x)=-\frac{\alpha}{2}\int_\R(x-y)_-^{\frac{\alpha}{2}-1}\nu(dy).
	$$
	Then
	$$
	f_{\nu,\alpha}(x)=\frac{\sin\theta}{\theta}\cdot\frac{A'B-AB'}{A^2+B^2+2AB\cos\theta}=-\frac{\Im(F'\bar{F})}{\theta|F|^2}=-\frac{1}{\theta}\Im\l(\frac{F'}{F}\r)=-\frac{1}{\theta}\frac{d}{dx}\arg F(x).
	$$
	Hence
	$$
	\int_\R f_{\nu,\alpha}(x)dx=-\frac{1}{\theta}\Big(\arg F(+\infty)-\arg F(-\infty)\Big).
	$$
	As $x\rw\infty$, $A(x)\sim x^\frac{\alpha}{2}$, $B(x)\rw 0$, then $\arg F(+\infty)= 0$. When $x\rw-\infty$, $A(x)\rw 0$, $B(x)\sim |x|^{\frac{\alpha}{2}}$, which gives $\arg F(-\infty)=\theta$. Hence $\int _\R f_{\nu,\alpha}(x)dx=1$.
	
	When $x\in K_\nu$, since $\nu$ is non-degenerate, both $\nu((x,+\infty))$ and $\nu((-\infty,x))$ are positive. As a consequence, the four integrals
	$$
	\int_\R(x-y)_+^{\frac{\alpha}{2}}\nu(dy),\ \int_\R(x-y)_-^{\frac{\alpha}{2}-1}\nu(dy),\ \int_\R(x-y)_-^{\frac{\alpha}{2}}\nu(dy),\ \int_\R(x-y)_+^{\frac{\alpha}{2}-1}\nu(dy),
	$$
	are positive and hence $f_{\nu,\alpha}(x)>0$. Finally, when $x\notin K_\nu$, at least one of $\nu((x,+\infty))$ or $\nu((-\infty,x))$ is 0. Then at least one of the following holds:
	$$
	\int_\R(x-y)_+^{\frac{\alpha}{2}}\nu(dy)=\int_\R(x-y)_+^{\frac{\alpha}{2}-1}\nu(dy)=0\quad\text{ or }\quad  \int_\R(x-y)_-^{\frac{\alpha}{2}}\nu(dy)=\int_\R(x-y)_-^{\frac{\alpha}{2}}\nu(dy)=0.
	$$
	This establishes $f_{\nu,\alpha}(x)=0$ and completes the proof.
\end{proof}
\begin{example}\label{ex:bounded}
	We conclude this subsection by giving two explicit examples of $\mu_{\nu,\alpha}$.
	In Section \ref{sec:verification-ex-bounded}, we check that these densities integrate to one and investigate the boundary cases $\alpha \to 0$ or $2$ to confirm Proposition \ref{prop:munualpha_examples} in these examples.
	\begin{enumerate}
		\item[$(i)$] If $\nu=\frac{1}{2}\delta_0+\frac{1}{2}\delta_1$, then
		\bea\nonumber
		f_{\nu,\alpha}(x)=\frac{\sin\l(\frac{\alpha\pi}{2}\r)}{\pi}\frac{(1-x)^{\frac{\alpha}{2}-1}x^{\frac{\alpha}{2}-1}}{x^\alpha+(1-x)^\alpha+2\cos\l(\frac{\alpha\pi}{2}\r)x^{\frac{\alpha}{2}}(1-x)^{\frac{\alpha}{2}}}\1_{(0,1)}(x).
		\eea
		\item[$(ii)$]  If $\nu(dx)=\1_{[0,1]}(x)dx$, then
		\bea\nonumber
		f_{\nu,\alpha}(x)=\frac{\sin\l(\frac{\alpha\pi}{2}\r)}{\pi}\frac{\alpha+2}{\alpha}\frac{(1-x)^{\frac{\alpha}{2}}x^{\frac{\alpha}{2}}}{x^{\alpha+2}+(1-x)^{\alpha+2}+2\cos\l(\frac{\alpha\pi}{2}\r)x^{\frac{\alpha}{2}+1}(1-x)^{\frac{\alpha}{2}+1}}\1_{(0,1)}(x).
		\eea
		
	\end{enumerate}
\end{example}

\subsection{Limiting law: unbounded diagonal case}

In this subsection, we extend Theorem~\ref{Prop-quadratic-form} by deriving the limiting law of $Q_{n,1}$ for sequences of random matrices $\A_n$ whose diagonal entries may be unbounded. The next result states that the conclusion of Theorem~\ref{Prop-quadratic-form} remains valid if the condition $M<\infty$ is replaced by \eqref{eq:unbounded} below. We remark that the condition \eqref{eq:unbounded} is satisfied when $a_{ii}^{(n)}, i\ge 1$, are iid with $\E\big| a_{11}^{(n)}\big|<\infty$.
\begin{theorem}\label{thm:qf-unbounded}
	Let $Q_{n,1}$ be as defined in \eqref{Def-Qn1Qn2} and $\alpha\in (0,2)$. Assume that $\frac{1}{n}\sum_{i=1}^n\delta_{a_{ii}^{(n)}}$ converges weakly almost surely to a deterministic probability measure $\nu$ and that 
	\begin{equation}\label{eq:unbounded}
		\lim_{T\rw\infty}\limsup_{n\rw\infty}\frac{1}{n}\sum_{i=1}^n|a_{ii}^{(n)}|\1(|a_{ii}^{(n)}|>T)=0 \qquad \as
	\end{equation}
	Then $Q_{n,1}$ converges in distribution to a random variable with induced probability measure $\mu_{\nu,\alpha}$ whose Stieltjes transform is given by \eqref{Eq-Stieltjes-transform-Quadratic-form}.
\end{theorem}
\begin{proof}[Proof of Theorem \ref{thm:qf-unbounded}]
	Analogously to the proof of Theorem \ref{Prop-quadratic-form}, we start with the case that $\{\A_n\}_{n=1}^\infty$ are deterministic. For $T>0$, we set $\A_n^{(T)}=\big(a_{ij}^{(n)}\1(|a_{ij}^{(n)}|\leq T)\big)_{i,j=1,\ldots,n}$ and use the decomposition $Q_{n,1}=Q_{n,1}^{(T)}+\wt{Q}_{n,1}^{(T)}
	$, where 
	$$Q_{n,1}^{(T)}=\mf{y}_1^\top \diag(\A_n^{(T)})\mf{y}_1,\quad \wt{Q}_{n,1}^{(T)}=\mf{y}_1^\top \diag(\A_n-\A_n^{(T)})\mf{y}_1\,.$$
	Since $\frac{1}{n}\sum_{i=1}^n\delta_{a_{ii}^{(n)}}\cw\nu$, 
	$$
	\frac{1}{n}\sum_{i=1}^n f(a_{ii}^{(n)})\1(|a_{ii}^{(n)}|\leq T)\rw\int_\R f(x)\1(|x|\leq T)\nu(dx),\quad \text{for any } f\in \ml{C}_b(\R).
	$$
	Then
	\bea\nonumber
	\frac{1}{n}\sum_{i=1}^n f(a_{ii}^{(n)}\1(|a_{ii}^{(n)}|\leq T))&=\frac{1}{n}\sum_{i=1}^n f(a_{ii}^{(n)})\1(|a_{ii}^{(n)}|\leq T)+\frac{1}{n}\sum_{i=1}^n f(0)\1(|a_{ii}^{(n)}|> T)\\ 
	&\rw \int_\R f(x)\1(|x|\leq T)\nu(dx)+f(0) \, \nu \big( x:|x|>T\big)\,.
	\eea
	Therefore for fixed $T>0$,
	$$
	\frac{1}{n}\sum_{i=1}^n \delta_{a_{ii}^{(n)}\1(|x|\leq T)}\cw\nu^{(T)},
	$$
	where $\nu^{(T)}$ is the probability measure such that
	$$
	\int_\R f(x)\nu^{(T)}(dx)=\int_\R f(x)\1(|x|\leq T)\nu(dx)+f(0)\, \nu \big( x:|x|>T\big),\qquad f\in \ml{C}_b(\R).
	$$
	Since $T<\infty$,  an application of Theorem \ref{Prop-quadratic-form} yields that, as $\nto$, the distribution of $Q_{n,1}^{(T)}$ converges to the probability measure $\mu_{\nu^{(T)},\alpha}$ with
	\bea\nonumber
	s_{\mu_{\nu^{(T)},\alpha}}(z)=-\frac{\int_{\R}(z-x)^{\frac{\alpha}{2}-1}\nu^{(T)}(dx)}{\int_{\R}(z-x)^{\frac{\alpha}{2}}\nu^{(T)}(dx)},\qquad z\in\C^+.
	\eea
	Clearly $\nu^{(T)}\cw\nu$ as $T\rw\infty$. 
	Since \eqref{eq:unbounded} implies
	$$\lim_{T\rw\infty}\limsup_{\nto} \frac{1}{n}\sum_{i=1}^n |a_{ii}^{(n)}|^{\frac{\alpha}{2}}\1(|a_{ii}^{(n)}|>T)=0,$$ we obtain by uniform integrability that $\int_\R|x|^{\frac{\alpha}{2}}\nu(dx)<\infty$. Then the dominated convergence theorem shows   
	$$
	\lim_{T\rw\infty}\int_{\R}(z-x)^{\frac{\alpha}{2}}\nu^{(T)}(dx)=\int_{\R}(z-x)^{\frac{\alpha}{2}}\nu(dx)\,,
	$$
	an in view of $|z-x|^{\frac{\alpha}{2}-1}\leq |\Im(z)|^{\frac{\alpha}{2}-1}$, it holds
	$$
	\lim_{T\rw\infty}\int_{\R}(z-x)^{\frac{\alpha}{2}-1}\nu^{(T)}(dx)=\int_{\R}(z-x)^{\frac{\alpha}{2}-1}\nu(dx).
	$$
	The last two displays imply $\lim_{T\to\infty}s_{\mu_{\nu^{(T)},\alpha}}(z)=s_{\mu_{\nu,\alpha}}(z)$ for all $z\in\C^+$, which is equivalent to 
	\begin{equation}\label{Eq-quadratic-form-general-case-cond2}\mu_{\nu^{(T)},\alpha}\cw\mu_{\nu,\alpha}.
	\end{equation}
	We also have
	\bea\nonumber
	\label{Eq-quadratic-form-general-case-cond3}\lim_{T\rw\infty}\limsup_{n\rw\infty}\E[|Q_{n,1}-Q_{n,1}^{(T)}|]\le \lim_{T\rw\infty}\limsup_{n\rw\infty}\frac{1}{n}\sum_{i=1}^n|a_{ii}^{(n)}|\1(|a_{ii}^{(n)}|>T)=0.
	\eea
	In conjunction with \eqref{Eq-quadratic-form-general-case-cond2}and the fact that the law of $Q_{n,1}^{(T)}$ converges weakly to $\mu_{\nu^{(T)}}$ as $\nto$, Lemma \ref{lemma-Billingsley-Thm2.3} completes the proof of convergence in distribution of $Q_{n,1}$.
\end{proof}
Clearly, Corollary~\ref{corr:mainthm} also remains valid if the condition $M<\infty$ is replaced by \eqref{eq:unbounded}. Next, recall that the density formula \eqref{Eq-Density-of-Quadratic-Form} holds for non-degenerate probability measures $\nu$. But when $\nu$ has unbounded support, $\mu_{\nu,\alpha}$ also has unbounded support. A natural question is: what is the tail probability of $\mu_{\nu,\alpha}$? Or equivalently, what is the tail behavior of $f_{\nu,\alpha}$? For illustration, we consider the case that $\nu$ is the exponential or Pareto distribution, respectively. The proof of the following example is deferred to the appendix.

\begin{example}\label{ex:unbounded}
	\begin{enumerate} 
		\item[$(i)$] When $\nu$ has exponential tail, $\mu_{\nu,\alpha}$ has a tail of a Gamma distribution with shape parameter $1-\alpha/2$ and rate $1$. More precisely, if $\nu(dx)=e^{-x}\1(x>0)dx$, then
		\bea\nonumber
		f_{{\nu,\alpha}}(x)\sim \frac{1}{\Gamma\l(1-\frac{\alpha}{2}\r)}x^{-\frac{\alpha}{2}}e^{-x},\quad x\rw\infty.
		\eea
		
		\item[$(ii)$] When $\nu$ has polynomial tail, $\mu_{\nu,\alpha}$ has the same polynomial tail with a tail index independent of $\alpha$, a situation very different from the previous one. More precisely, if $\nu(dx)=sx^{-s-1}\1(x>1)dx,\ s>\frac{\alpha}{2}$, then
		\bea\label{Eq-Quadratic-form-nu-pareto-tail-probability}
		f_{{\nu,\alpha}}(x)\sim \frac{s}{\pi}\sin\l(\frac{\alpha\pi}{2}\r)\l(B\l(\frac{\alpha}{2},s+1-\frac{\alpha}{2}\r)+B\l(\frac{\alpha}{2}+1,s-\frac{\alpha}{2}\r)\r)x^{-s-1},\quad x\rw\infty,
		\eea
		where $B(x,y)$ is the Beta function.
	\end{enumerate}
\end{example}

%%%%%%%%%%%%%%%%%%%%%%%%%%%%%%%%%%%%%%%%%%%%
\section{Application in RMT: $\alpha$-heavy MP law}
\label{sec:app-RMT}

As explained in the introduction, the limiting spectrum of a sample correlation matrix is described 
by the celebrated \MP\ (MP) law when the initial entries $(X_{ij})$ have a finite variance. By considering regularly
varying variables $X_{ij}$ with index $\alpha\in(0,2)$, Heiny and Yao \cite{Heiny:Yao:AOS} first established a new limiting spectral 
distribution $H_{\alpha,\gamma}$, called the $\alpha$-heavy MP law. This limit is derived via the method of 
moments, so that we only know the sequence of moments of $H_{\alpha,\gamma}$. Analytical properties of 
$H_{\alpha,\gamma}$ like its support, its absolute continuity or singular components (with respect to the 
Lebesgue measure), or its tail behavior at infinity (assuming an unbounded support) remain open. 
In this section, using the main theorems developed in the previous section, we advance toward these
open questions by proving that except for a possible mass at the origin, the the $\alpha$-heavy \MP law has no atoms. The analysis of the $\alpha$-heavy \MP law is different from the classical ones, due to the fact that the resolvent is no longer approximated by a constant diagonal matrix that appears in the classical local law. Instead, the resolvent should be approximated by a random diagonal matrix. Although the limiting law of the resolvent's diagonal is hard to solve, we can still derive the no-atom property thanks to the continuity property of the limiting distribution of the quadratic forms $Q_{n,1}$; see Theorem~\ref{thm::quadratic-from-property}.
\medskip

Recall that  $\y_1, \ldots, \y_p \in \R^n$  are the columns of the matrix $\Y^{\top}=\Y_n^{\top}$ in \eqref{eq:defRY_new}.
Throughout this section, we will assume the  \emph{high-dimensional proportional regime}
\[
p = p_n \to \infty \quad \text{ and } \quad \frac{p}{n} \to \gamma \in (0, \infty)\,,\qquad \text{as} \quad n \to \infty.
\]
Fix \( z \in \C^+ \) and define the matrix resolvent
$$
\B_n(z) := \left( \Y_n^\top \Y_n - z \bfI_n \right)^{-1} =  \left( \sum_{i=1}^p \y_i \y_i^\top - z \bfI_n \right)^{-1} \in \mathbb{C}^{n \times n}.
$$
Let \( \nu_{z,n} \) denote the empirical measure of the diagonal entries of \( \B_n(z) \), that is,
$$
\nu_{z,n} := \frac{1}{n} \sum_{j=1}^n \delta_{\B_n(z)_{jj}}.
$$
Note that $\nu_{z,n}$ is a probability measure with bounded support inside $H_z=\{\zeta\in\C:|\zeta|\leq1/\operatorname{dist}(z,\R^+)\}$. In the light-tailed regime where the local law holds, we have $\mf{B}_n(z)_{11}-\E[\mf{B}_n(z)_{11}]\cas 0$, hence $(\nu_{z,n})_n$ converges weakly \as\ to a Dirac measure. However, in the heavy-tailed regime, $(\nu_{z,n})_n$ might have a non-degenerate weak limit point. We first show that \( \nu_{z,n} \), as a random measure, concentrates around its expectation only using the independence and unit length of $\y_i$'s.

\begin{proposition}\label{Prop-concentration-resolvant-diagonal}
	Let \( f : \mathbb{C} \to \mathbb{\C} \) be an $L_f$-Lipschitz function. Then for any \( t > 0 \) and $z\in \C\backslash[0,+\infty)$,
	\bea\nonumber
	\P\l(\l|\frac{1}{n} \sum_{j=1}^n f\big( \B_n(z)_{jj} \big)-\E\l[\frac{1}{n} \sum_{j=1}^n f\big( \B_n(z)_{jj} \big)\r]\r|>t\r)\leq2\exp\l(-\frac{n^2t^2\operatorname{dist}(z,\R^+)}{8pL_f^2}\r).
	\eea
\end{proposition}
The proof of Proposition~\ref{Prop-concentration-resolvant-diagonal} relies on a bounded martingale difference argument and can be found in the appendix.
Proposition \ref{Prop-concentration-resolvant-diagonal} and the Borel--Cantelli lemma yield for any Lipschitz function $f$ and $z\in \C\backslash[0,+\infty)$,
$$
\int_{\C} f(x)\nu_{z,n}(dx)-\E\l[\int_\C f(x)\nu_{z,n}(dx)\r] \cas  0 \,, \qquad \nto.
$$
Note that $\{\E[\nu_{z,n}]\}_n$, as a sequence of deterministic probability measures, is also supported inside $H_z$. By Prokhorov's theorem, for every $z\in\C\backslash[0,+\infty)$ the sequence  $\{\E[\nu_{z,n}]\}_n$ has a weakly convergent subsequence. For any Lipschitz function $f$, we have
$$
\E[f(\B_n(z)_{11})]=\frac{1}{n}\sum_{j=1}^n \E[f(\B_n(z)_{jj})]=\int_{\C}f(x)\E[\nu_{z,n}](dx).
$$ Hence a subsequence of $\{\B_n(z)_{11}\}_n$ converges in distribution if and only if the same subsequence of $\{\E[\nu_{z,n}]\}_n$ converges weakly to the same limit. For any set $D\subset \C$, we denote $\ell_\infty(D)$ as the Banach space of all uniformly bounded complex functions on $D$ equipped with the sup-norm.
\begin{lemma}\label{lemma-uniform-tightness-of-resolvant-diagonal}
	There exists a random holomorphic function $\psi:\C\backslash[0,+\infty)\rw \C$ and a subsequence $(n_k)_{k=1}^\infty$ such that, for any compact set $D\subset \C\backslash[0,+\infty)$, $\{\B_{n_k}(\cdot)_{11}|_D\}_{k=1}^\infty$ converges weakly  to $\psi|_D$ as random elements valued in the Banach space $\ell_\infty(D)$.
\end{lemma}
\begin{proof}
	Writing
	$K_m=\{z\in \C:\ \operatorname{dist}(z,\R^+)\geq\frac{1}{m},\ |z|\leq m\}, m\in \N$, 
	one has that $\sup_{z\in K_m}|\B_n(z)_{11}|\leq m$ and for $z_1,z_2\in K_m$,
	\begin{align*}
		|\B_{n}(z_1)_{11}-\B_{n}(z_2)_{11}|&=|\e_1^\top[(\Y_n^\top \Y_n-z_1 \bfI_n)^{-1}-(\Y_n^\top \Y_n-z_2 \bfI_n)^{-1}]\e_1|\\
		&=|\e_1^\top(\Y_n^\top \Y_n-z_2 \I_n)^{-1}(z_2-z_1)(\Y_n^\top \Y_n-z_1 \I_n)^{-1}\e_1|\\
		&\leq m^2|z_2-z_1|.
	\end{align*}
	Then by Theorem 1.5.7 in \cite{Book-van-der-Vaart-weak-convergence}, $\{\B_n(\cdot)_{11}\}_n$ is an asymptotically tight sequence in the Banach space $\ell_\infty(K_m)$. By Prokhorov's theorem, there is a subsequence $(n_{1,k})_{k=1}^\infty$ such that $\{\B_{n_{1,k}}(\cdot)_{11}|_{K_1}\}_{k=1}^\infty$ converges weakly to a random element $B_{1,\infty}(\cdot)$ valued in $\ell_{\infty}(K_1)$. Then again by Prokhorov's theorem, we can choose a further subsequence $(n_{2,k})_{k=1}^\infty\subset (n_{1,k})_{k=1}^\infty$ such that $\{\B_{n_{2,k}}(\cdot)_{11}|_{K_2}\}_{k=1}^\infty$ converges weakly to a random element $B_{2,\infty}(\cdot)$ valued in $\ell_{\infty}(K_2)$. Then $B_{2,\infty}(\cdot)|_{K_1}=B_{1,\infty}(\cdot)$. Repeating this argument, we obtain a double-labeled subsequence $(n_{m,k})_{m,k=1}^\infty$ and random elements $\{B_{m,\infty}\}_{m=1}^\infty$ such that for all $m\in \N$,
	$$
	(n_{m+1,k})_{k=1}^\infty\subset (n_{m,k})_{k=1}^\infty,\qquad B_{m+1,\infty}(\cdot)|_{K_m}=B_{m,\infty}(\cdot),
	$$
	and $\{\B_{n_{m,k}}(\cdot)_{11}|_{K_m}\}_{k=1}^\infty$ converges weakly to $B_{m,\infty}(\cdot)$. It suffices to choose $(n_k)=(n_{k,k})$, and $\psi(\cdot)=\lim_{m\rw\infty}B_{m,\infty}(\cdot)$ (in the sense of point-wise limit). 
	
	It remains to prove that $\psi(\cdot)$ is holomorphic on $\C\backslash[0,+\infty)$. Let $\eta$ be any closed and piecewise $C^1$ curve on $\C\backslash[0,+\infty)$, then there is $m\in \N$ such that $\eta\subset K_m$. Consider the functional 
	$$
	J_\eta[f]: \ell_\infty(K_m)\rw \C,\quad f\mapsto \oint_{\eta} f(z)dz.
	$$
	Then $J_\eta$ is bounded and continuous. Then by weak convergence and continuous mapping theorem,
	$
	J_\eta[\B_{n_k}(\cdot)_{11}|_{K_m}]\rw J_\eta[\psi(\cdot)]
	$ as $k \to\infty$.
	Given $\Y_n$, $\B_{n_k}(\cdot)|_{11}$ is analytic on $\C\backslash[0,+\infty)$. Then $J_\eta[\B_{n_k}(\cdot)_{11}|_{K_m}]=0$ from Cauchy's integral formula. Hence $J_\eta[\psi(\cdot)]=0$. Since $\eta$ can be arbitrary, by Morera's theorem \cite{Book-Stein-Complex-analysis}, $\psi(\cdot)$ is holomorphic on $\C\backslash[0,+\infty)$.
\end{proof}

% \begin{proposition}
	%     There exist holomorphic functions $\phi(z),\psi(z)$ on $\C^+$ satisfying the followings. 
	%     \begin{enumerate}
		%         \item There exist a subsequence $(p_k,n_k)_{k=1}^\infty$ such that for all $z\in \C^+$,
		%         \bea
		%         \phi(z)=\lim_{k\rw\infty} \E[(1+B_{p_k}(z)_{11})^{\frac{\alpha}{2}}],
		%         \eea
		%         and for all $z\in \C,\ \Re z<0$,
		%         \bea
		%         \psi(z)=\lim_{k\rw\infty} \E[(1+B_{p_k}(z)_{11})^{\frac{\alpha}{2}-1}].
		%         \eea
		%         \item For all $z\in \C^+$,
		%         \bea
		%         s_{H_{\alpha,\gamma}}(z)=-\frac{\psi(z)}{z\phi(z)}.
		%         \eea
		%     \end{enumerate}
	
	% \end{proposition}

It turns out that this random function $\psi$, together with the quadratic form results in Theorem \ref{Prop-quadratic-form}, facilitates an elegant expression for the Stieltjes transform of the $\alpha$-heavy MP law $H_{\alpha,\gamma}$.
\begin{theorem}\label{prop-aHeavyMP-Stieltjes-randomholomorphic-embedding}
	Let $\psi$ be the random holomorphic function in Lemma \ref{lemma-uniform-tightness-of-resolvant-diagonal}, then for $0<\alpha<2$ and $\gamma>0$, it holds
	\bea\label{eq-aheavy-embedding}
	s_{H_{\alpha,\gamma}}(z)=-\frac{\E[(1+\psi(z))^{\frac{\alpha}{2}-1}]}{z\E[(1+\psi(z))^{\frac{\alpha}{2}}]},\qquad \ z\in \C\backslash[0,+\infty).
	\eea
\end{theorem}
We will first prove \eqref{eq-aheavy-embedding} for $z\in (-\infty,0)$, where $\psi(z)$ is real-valued and Theorem \ref{Prop-quadratic-form} is applicable. After that, we will prove that \eqref{eq-aheavy-embedding} holds for all $z\in \C\backslash(0,+\infty)$ by holomorphic extension.
\begin{proof}[Proof of Theorem \ref{prop-aHeavyMP-Stieltjes-randomholomorphic-embedding}]
	Let $z\in\C\backslash[0,+\infty)$. By \cite[Theorem 2.5]{doernemann:heiny:2025}, we know that
	\begin{equation}\label{eq:thm2.5}
		\lim_{n\rw\infty}\E[s_{\Y_n\Y_n^\top}(z)]=-\frac{1}{z}\lim_{n\rw\infty}\E\l[\frac{1}{1+W_n(z)+\gamma\E[s_{\Y_n\Y_n^\top}(z)]-z^{-1}(1-\gamma)}\r],
	\end{equation}
	where $W_n(z)=\y_1 {\B_{n,1}(z)} \y_1^\top-\frac{1}{n}\tr\B_{n,1}(z)$ with
	$$
	\B_{n,k}(z)=(\Y_n^\top \Y_n-\y_k\y_k^\top-z\I_n)^{-1}\,, \qquad k=1,\ldots, p.
	$$
	It is proved in \cite{Heiny:Yao:AOS} that
	$s_{\Y_n\Y_n^\top}(z)\cas s_{H_{\alpha,\gamma}}(z)$, as $\nto$, and therefore
	$$
	\lim_{n\rw\infty}\frac{1}{n}\tr \B_{n,1}(z)=\gamma s_{H_{\alpha,\gamma}(z)}-z^{-1}(1-\gamma)\quad \as
	$$
	In combination with \eqref{eq:thm2.5}, this yields
	\bea\nonumber
	s_{H_{\alpha,\gamma}}(z)=-\frac{1}{z}\lim_{n\rw\infty}\E\l[\frac{1}{1+\y_1^\top\B_{n,1}(z)\y_1}\r],\quad z\in \C\backslash[0,+\infty).
	\eea    
	Clearly, the results of Proposition \ref{Prop-concentration-resolvant-diagonal} and Lemma \ref{lemma-uniform-tightness-of-resolvant-diagonal} remain true when replacing $\B_n$ with $\B_{n,1}$. 
	
	Now let $\psi$ and the subsequence $(n_k)_{k=1}^\infty$ be as given in Lemma \ref{lemma-uniform-tightness-of-resolvant-diagonal}. For $t\in (-\infty,0)$,  the matrix $\B_{n,1}(t)$ is Hermitian. Since $\B_{n,1}(t)$ is a resolvent matrix, by Example \ref{ex:resolvent}, Assumption \ref{ass:offdiag} is satisfied. Then by Lemma \ref{Prop-quadratic-form-Qn2}, $$\y_1^\top \B_{n,1}(t)\y_1-\y_1^\top {\diag(\B_{n,1}(t))}\y_1\cip 0.$$ By Proposition \ref{Prop-concentration-resolvant-diagonal} and Lemma \ref{lemma-uniform-tightness-of-resolvant-diagonal}, the counting measure $n^{-1}\sum_{j=1}^n \delta_{\B_{n,1}(t)_{jj}}$ converges weakly almost surely to the law of $\psi(t)$, say $\nu_t$. By Theorem \ref{Prop-quadratic-form}, $\y_1^\top {\diag(\B_{n,1}(t))}\y_1$ (and also $\y_1^\top \B_{n,1}(t)\y_1$) converges in distribution to $\mu_{\nu_t,\alpha}$, where
	$$
	s_{\mu_{\nu_t,\alpha}}(\zeta)=-\frac{\int_\R(\zeta-x)^{\frac{\alpha}{2}-1}\nu_t(dx)}{\int_\R(\zeta-x)^{\frac{\alpha}{2}}\nu_t(dx)}=-\frac{\E[(\zeta-\psi(t))^{\frac{\alpha}{2}-1}]}{\E[(\zeta-\psi(t))^{\frac{\alpha}{2}}]},\quad \zeta\in \C\backslash\operatorname{supp}(\psi(t)).
	$$
	% \jh{I think we should use some connection with $\nu_{z,n}$ to make this step easier to follow.}\zd{ $\phi(z)$ is a limit point of $\nu_{z,n}$. Here we already let $n\rw\infty$, so $\nu_{z,n}$ is not involved. Everything reduces a problem of complex analysis. The idea is that, we need to do Theorem \ref{prop-aHeavyMP-Stieltjes-randomholomorphic-embedding} firstly for $z$ on negative real line, and then extend $z$ to all $\C\backslash[0,+\infty)$ via holomorphic extension. That's the reason why Proposition \ref{Prop-concentration-resolvant-diagonal} needs to hold for $z\in \C\backslash[0,+\infty)$, and in Lemma \ref{lemma-uniform-tightness-of-resolvant-diagonal} we need $\psi$ to be holomorphic on the whole $\C\backslash[0,+\infty)$, but not only $\C^+$. The reason that we have to start from $z$ on negative real line is that, in Theorem \ref{Prop-quadratic-form} we have to use moment method, then the quadratic form must be real (moments can not characterize complex probability measures), hence the diagonals must be real. The choice of using $\C\backslash[0,+\infty)$ instead of $\C$ involves sophisticated technical concerns through the paper. Therefore, I strongly recommend to put all the $\C\backslash[0,+\infty)$ back.}
	Since $\B_{n,1}(t)_{11}\geq 0$, $\psi(t)\geq 0$, and letting $\zeta=-1$, we have
	$$
	s_{H_{\alpha,\gamma}}(t)=-\frac{1}{t}s_{\mu_{\nu_t,\alpha}}(-1)=-\frac{\E[(1+\psi(t))^{\frac{\alpha}{2}-1}]}{t\E[(1+\psi(t))^{\frac{\alpha}{2}}]},\qquad t\in (-\infty,0).
	$$
	Since $|\psi(z)|\leq (\operatorname{dist}(z,\R^+))^{-1}$, by Morera's theorem and Fubini's theorem, $\E[(1+\psi(z))^{\frac{\alpha}{2}}]$ is holomorphic on $\C\backslash[0,+\infty)$. When $\Re z\leq 0$, we can use the spectral decomposition $\Y_n^\top \Y_n {-\y_1\y_1^\top} =\sum_{j=1}^n \lambda_j\bfv_j\bfv_j^\top$ to obtain
	$$
	\Re (1+\B_{n,1}(z)_{11})=1+\sum_{j=1}^n\frac{(\lambda_j-\Re z)(\e_1^{\top}\bfv_j)^2}{|\lambda_j-z|^2}\geq 1,
	$$
	from which we conclude 
	\bea\label{eq-bound-1+psi-Rez<0}
	\inf_{\Re z\leq 0}|1+\psi(z)|\geq 1.
	\eea
	Hence $\E[(1+\psi(z))^{\frac{\alpha}{2}-1}]$ is holomorphic when $\Re z<0$. 
	When $\Re z\geq 0$, by Cauchy--Schwartz inequality,
	$$
	\l(\sum_{j=1}^n \frac{(\e_1^\top \bfv_j)^2}{|\lambda_j-z|^2}\r)\l(\sum_{j=1}^n (\e_1^\top \bfv_j)^2|\lambda_j-z|^2\r)\geq \l(\sum_{j=1}^n (\e_1^\top \bfv_j)^2\r)^2=1,
	$$
	and we have
	\begin{align*}
		\sum_{j=1}^n (\e_1^\top \bfv_j)^2 |\lambda_j-z|^2&=
		\sum_{j=1}^n (\e_1^\top \bfv_j)^2 \l((\lambda_j-\Re z)^2+ (\Im z)^2\r)\\
		&=|z|^2+\sum_{j=1}^n(\e_1^\top \bfv_j)^2\lambda_j^2-2\Re z \sum_{j=1}^n (\e_1^\top \bfv_j)^2\lambda_j\\
		&\leq |z|^2+(\Y_n^\top \Y_n \Y_n^\top \Y_n)_{11},
	\end{align*}
	and hence 
	\bea\nonumber
	|\Im (1+\B_{n,1}(z)_{11})|&=|\Im z|\sum_{j=1}^n (\e_1^\top \bfv_j)^2 |\lambda_j-z|^2\\
	&\geq |\Im z|(|z|^2+(\Y_n^\top \Y_n \Y_n^\top \Y_n)_{11})^{-1}.
	\eea
	Consequently, $(1+\B_n(z)_{11})^{\frac{\alpha}{2}-1}$ is holomorphic when $\Re z>0$, and
	\bea\label{eq-bound-1+psi-Rez>0}
	|1+\B_{n}(z)_{11}|^{\frac{\alpha}{2}-1}\leq \frac{1}{|\Im z|^{1-\frac{\alpha}{2}}}(|z|^2+(\Y_n^\top \Y_n \Y_n^\top \Y_n)_{11})^{1-\frac{\alpha}{2}}.
	\eea
	Then
	\begin{align*}
		\E[|1+\psi(z)|^{\frac{\alpha}{2}-1}]&=\lim_{k\rw\infty}\E[|1+\B_{n_k}(z)_{11}|^{\frac{\alpha}{2}-1}]\\
		&\leq \frac{1}{|\Im z|^{1-\frac{\alpha}{2}}}\l(|z|^2+1+\lim_{n\rw\infty}\E[(\Y_n^\top \Y_n\Y_n^\top \Y_n)_{11}]\r)\\
		&=\frac{1}{|\Im z|^{1-\frac{\alpha}{2}}}\l(|z|^2+1+\int_{\R}x^2H_{\alpha,\gamma}(dx)\r).
	\end{align*}
	By (2.3) in \cite{Heiny:Yao:AOS}, $\int_{\R}x^2H_{\alpha,\gamma}(dx)<\infty$. Now let $\eta$ be any closed and piecewise $C^1$ curve in $\{\zeta\in \C^+:\Re \zeta >0\}$ such that $z$ is in interior of $\eta$. Then
	\begin{align*}
		\oint_\eta \E[(1+\psi(z))^{\frac{\alpha}{2}-1}] dz&=\E\l[\oint_\eta(1+\psi(z))^{\frac{\alpha}{2}-1}dz\r]\quad \text{(Fubini)}\\
		&=\E\l[\lim_{k\rw\infty}\oint_\eta (1+\B_{n_k}(z)_{11})^{\frac{\alpha}{2}-1}dz\r]\\
		&=0.
	\end{align*}
	Then by Morera's theorem, $\E[|1+\psi(z)|^{\frac{\alpha}{2}-1}]$  is holomorphic on $\{\zeta\in \C^+:\Re \zeta >0\}$. By combining the bound \eqref{eq-bound-1+psi-Rez<0},\eqref{eq-bound-1+psi-Rez>0} and the same procedure, $\E[|1+\psi(z)|^{\frac{\alpha}{2}-1}]$ is also holomorphic on $\{\zeta\in \C^+:\Re \zeta=0\}$.
	Thus, $\E[|1+\psi(z)|^{\frac{\alpha}{2}-1}]$ is holomorphic on $\C\backslash[0,+\infty)$. 
	
	Now consider 
	$$
	F(z)=s_{H_{\alpha,\gamma}}(z)+\frac{\E[(1+\psi(z))^{\frac{\alpha}{2}-1}]}{z\E[(1+\psi(z))^{\frac{\alpha}{2}}]},\quad z\in \C\backslash[0,+\infty),
	$$
	then $F(z)$ is holomorphic on $\C\backslash[0,+\infty)$ and is zero on $(-\infty,0)$. We conclude that $F(z)\equiv 0$ on $\C\backslash[0,+\infty)$, which finishes the proof.
\end{proof}

\begin{remark}
	Although Theorem \ref{prop-aHeavyMP-Stieltjes-randomholomorphic-embedding} holds only for $0<\alpha<2$, we can still recover the \MP law by letting $\alpha\uparrow2$. Taking $\alpha=2$ informally in Theorem \ref{prop-aHeavyMP-Stieltjes-randomholomorphic-embedding}, write $s(z)=s_{H_{\alpha,\gamma}}(z)$ for short, we have
	\be\label{eq-randomhol-embedding-alpha=2}
	s(z)=-\frac{1}{z\E[1+\psi(z)]}.
	\ee
	Assume the local law which formally needs some extra conditions holds. It says that 
	$$\mf{B}_n(z)_{11}-s_{\mf{Y}_n^\tp\mf{Y}_n}(z)\cas 0.$$ 
	Note that 
	$
	s_{\mf{Y}_n^\tp\mf{Y}_n}(z)\cas -\frac{1-\gamma}{z}+\gamma s(z)
	$.
	Hence the law of $\psi(z)$ is exactly the Dirac measure at $-\frac{1-\gamma}{z}+\gamma s(z)$. Then \eqref{eq-randomhol-embedding-alpha=2} becomes
	$$
	s(z)=-\frac{1}{z-1+\gamma+\gamma zs(z)},
	$$
	which we recognize as the Stieltjes transform of the \MP law. 
\end{remark}

Finally, we show that $H_{\alpha,\gamma}$ has no point masses on $(0,+\infty)$. 
\begin{proposition}\label{prop:nopoint}
	If $u>0$, then $H_{\alpha,\gamma}(\{u\})=0$.
\end{proposition}
\begin{proof}[Proof of Proposition \ref{prop:nopoint}]
	The proof strategy is to utilize \eqref{eq-aheavy-embedding} for $z\in \C^+$. If there were an atom at $\Re z$, then the two sides of \eqref{eq-aheavy-embedding} would have different asymptotic behavior as $\Im z\downarrow 0$, which is impossible.
	
	To arrive at a contradiction, assume that there exists some $u>0$ such that $H_{\alpha,\gamma}(\{u\})>0$. Let $z=u+iv$ with $0<v\leq 1$. From Theorem \ref{prop-aHeavyMP-Stieltjes-randomholomorphic-embedding}, 
	\bea\nonumber
	\E[(1+\psi(z))^{\frac{\alpha}{2}}]=-\frac{1}{z}\cdot \frac{(z-u)\E[(1+\psi(z))^{\frac{\alpha}{2}-1}]}{(z-u)s_{H_{\alpha,\gamma}}(z)}.
	\eea
	In view of \eqref{eq-bound-1+psi-Rez>0}, we have, as $v\downarrow0$,
	$\E[|1+\psi(z)|^{\frac{\alpha}{2}-1}]\lesssim v^{\frac{\alpha}{2}-1}$ and consequently
	$$
	|(z-u)\E[(1+\psi(z))^{\frac{\alpha}{2}-1}]|\lesssim v^{\frac{\alpha}{2}}\rw 0.
	$$
	Since
	$(z-u)s_{H_{\alpha,\gamma}}(z)\rw H_{\alpha,\gamma}(\{u\})>0$, this implies
	\bea\label{eq-atomproof-to-contradict-with}
	\lim_{v\downarrow0} \E[(1+\psi(z))^{\frac{\alpha}{2}}]=0.
	\eea
	Now let $\wt{H}_{\alpha,\gamma}$ be the LSD of $\Y_n^\top \Y_n$. Then $$w:=\wt{H}_{\alpha,\gamma}(\{u\})= (1-\gamma) + \gamma {H}_{\alpha,\gamma}(\{u\})>0.$$
	In view of 
	$
	\E[\Im \psi(z)]= \Im s_{\wt{H}_{\alpha,\gamma}}(z)\sim \frac{{w}}{v}$ as $ v\downarrow 0$,
	we get for $v$ sufficiently small that $\mathbb E[\Im\psi(z)] \ge \frac{w}{2v}$.
	
	By definition of $\B_n(z)$, we have $|\B_n(z)_{11}|\leq v^{-1}$ and $\Im \B_n(z)_{11}>0$. Then by construction of $\psi(z)$ in Lemma \ref{lemma-uniform-tightness-of-resolvant-diagonal},
	$$
	|\psi(z)|\leq v^{-1},\qquad \Im\psi(z)>0.
	$$
	Writing
	$
	A_v := \left\{\Im\psi(z)\ge \frac{w}{4v}\right\},
	$ we obtain
	$$
	\mathbb E[\Im\psi(z)]
	\le \frac{w}{4v}\,\mathbb P(A_v^c)+\frac{1}{v}\,\mathbb P(A_v)
	= \frac{w}{4v}+\frac{1-\frac w4}{v}\,\mathbb P(A_v).
	$$
	Rearranging yields
	$$
	\mathbb P(A_v)\ge \frac{v\,\mathbb E[\Im\psi(z)]-\frac w4}{1-\frac w4}
	\ge \frac{\frac w2-\frac w4}{1-\frac w4}
	=:p_0>0
	$$
	for all sufficiently small $v$. In particular, $\mathbb P(A_v)$ is bounded below by a positive constant independent of $v$.
	
	On $A_v$ we have 
	$$
	|1+\psi(z)| \ge \Im(1+\psi(z)) \ge \Im (\psi(z))\ge \frac{w}{4v}.
	$$
	Write $1+\psi(z)=re^{i\theta}$ with $\theta\in(0,\pi)$ (since $\Im(1+\psi(z))>0$), and $\beta=\alpha/2\in(0,1)$. Then
	$$
	\Im\l[(1+\psi(z))^\beta\r] = r^\beta \sin(\beta\theta).
	$$
	Furthermore,
	$$
	\sin\theta=\frac{\Im(1+\psi(z))}{|1+\psi(z)|}
	\ge \frac{\Im(1+\psi(z))}{1+|\psi(z)|}
	\ge \frac{\frac{w}{4v}}{\frac{2}{v}}=\frac{w}{8},
	$$
	so $\theta\ge \theta_0:=\arcsin(w/8)>0$ and hence $\sin(\beta\theta)\ge \sin(\beta\theta_0)=:s_0>0$. Therefore, on $A_v$,
	$$
	\Im\l[(1+\psi(z))^\beta\r]
	\ge \left(\frac{w}{4v}\right)^\beta s_0
	=: C\,v^{-\beta},
	$$
	where $C=(w/4)^\beta s_0>0$ depends only on $w$ and $\beta$.
	
	Combining the above bounds, we have
	$$
	\mathbb E\l[\Im\l[(1+\psi(z))^\beta\r]\r]
	\ge \mathbb E\l[\Im\l[(1+\psi(z))^\beta \r]\mathbf 1_{A_v}\r]
	\ge C\,v^{-\beta}\,\mathbb P(A_v)
	\ge C\,p_0\,v^{-\beta}.
	$$
	Hence $\lim_{v\downarrow0} \Im\E[(1+\psi(z))^{\frac{\alpha}{2}}]=+\infty$, a contradiction to \eqref{eq-atomproof-to-contradict-with}. Therefore, we conclude that $H_{\alpha,\gamma}(\{u\})=0$ for any $u>0$ and finish the proof.
\end{proof}

\subsection{Slowly varying tails ($\alpha=0$)}

On the one hand, Proposition~\ref{prop:nopoint} shows that for $\alpha\in (0,2)$, the $\alpha$-heavy MP law $H_{\alpha,\gamma}$ has no point mass on the positive real line. On the other hand, \cite{Heiny:Yao:AOS} proves that the measure $H_{0,\gamma}$, defined as the $\lim_{\alpha \to 0} H_{\alpha,\gamma}$, is a zero-inflated Poisson distribution  given by
\begin{equation}\label{eq:zeroinflated}
	H_{0,\gamma}(\{0\})=1-\gamma^{-1}+\gamma^{-1}e^{-\gamma} \quad \text { and } \quad H_{0,\gamma}(\{k\})=\gamma^{-1}e^{-\gamma}\frac{\gamma^k}{k!},\quad k\in \N.
\end{equation}
This means that the limit of continuous measures is a discrete one. It is worth pointing out that the LSD of $\bfR$ for $\alpha=0$ has not been investigated in the literature. The main goal of this subsection is to prove that this LSD is indeed $H_{0,\gamma}$.

Throughout this subsection, $\xi$ is assumed to be slowly varying, that is, $\xi$ satisfies the regular variation condition \eqref{eq:regvar} with index $\alpha=0$. Proposition 3 in \cite{mason:zinn:2005} asserts that $\xi$ is slowly varying if and only if $n\beta_4 \to 1$.

It turns out that for slowly varying $\xi$ the random vectors $\y_i=(Y_{i1},\ldots, Y_{in})^\top, i=1,\ldots, p$ have a strikingly simple structure. For its description
we need the notation
$$k_i:=\underset{k=1,\ldots,n}{\arg \max} \,|Y_{ik}|\,$$
and recall that $\e_i\in \R^n$ is the unit vector with $i$-th component $1$.

\begin{lemma}\label{lem:structure}
	For slowly varying $\xi$ it holds $\twonorm{\y_i-\sign(Y_{ik_i}) \e_{k_i}}\cip 0$ for $i=1,\ldots,p$.
\end{lemma}
\begin{proof}
	Since $\twonorm{\y_i}=1$, we have 
	$$\twonorm{\y_i-\sign(Y_{ik_i}) \e_{k_i}}^2 =(1-Y_{ik_i}^2)+ (Y_{ik_i}-\sign(Y_{ik_i}))^2=2(1-|Y_{ik_i}|)\,.$$
	For $\vep \in(0,1)$, consider the event $E_n=\{\exists k\neq \ell : |Y_{1k}|>\vep, |Y_{1\ell}|>\vep\}$.
	On $E_n$ we have $\sum_{k=1}^n Y_{1k}^4\le (1-\vep)^4+\vep^4$. Since $\sum_{k=1}^n Y_{1k}^4\le 1$ and $n\beta_4\to 1$, this implies $\P(E_n)\to 0$ and in turn $|Y_{1k_1}|\cip 1$ because $\twonorm{\y_1}=1$.
\end{proof}
Next, we extend Theorem~\ref{Prop-quadratic-form} to the case $\alpha=0$. 
\begin{proposition}\label{prop:extension}
	If $M:=\limsup_{n\rw\infty}\max_{1\leq i\leq n}|a_{ii}^{(n)}|<\infty$ almost surely, $\frac{1}{n}\sum_{i=1}^n\delta_{a_{ii}^{(n)}}$ converges weakly almost surely to a deterministic probability measure $\nu$ and $\xi$ is slowly varying ($\alpha=0$), then the distribution of $Q_{n,1}$ converges weakly to $\nu$.
\end{proposition}
\begin{proof}
	An application of Lemma~\ref{lem:structure} and using $M<\infty$ shows that
	\begin{align*}
		Q_{n,1}&=\y_1^\top \diag(\A_n) \y_1 = \e_{k_1}^\top \diag(\A_n) \e_{k_1} + o_{\P}(1)
	\end{align*}
	Note that $\e_{k_1}^\top \diag(\A_n) \e_{k_1}=a_{k_1k_1}$ and by construction, $k_1$ is uniformly distributed on $\{1,\ldots,n\}$.
	For a function $g\in \ml{C}_b(\R)$ we have
	$$\E[g(a_{k_1k_1})]=\frac{1}{n} \sum_{i=1}^n g(a_{ii}) \to \int g\, \nu(dx)\,, \qquad \nto\,,$$
	establishing the desired result in view of Slutsky's theorem.
\end{proof}
The construction in the proof shows how the limiting measure $\nu$ appears for $Q_{n,1}$. It is worth noting that Proposition~\ref{prop:extension} is consistent with Proposition~\ref{prop:munualpha_examples} (ii).

In view of Lemma~\ref{lem:structure}, it is natural to ask what happens if the vectors $\y_i$ in $\bfR_n=\Y_n \Y_n^\top$ are replaced by $\sign(Y_{ik_i}) \e_{k_i}$.
Next, we study the spectral distribution of the matrix $\widetilde{\bfR}_n=\bfZ_n\bfZ_n^\top$, where $\bfZ_n^\top$ is an $n\times p$ matrix, whose columns $\bfz_1,\ldots,\bfz_p$ are given by $\bfz_i:=\sign(Y_{ik_i}) \e_{k_i}$

\begin{proposition}\label{prop:Rtilde}
	Let $\widetilde{\bfR}_n$ be as defined above, then as $p,n\rw\infty,\ p/n\rw\gamma\in(0,+\infty)$, $F^{\widetilde{\bfR}_n}$ converges weakly in probability to the zero-inflated Poisson distribution $H_{0,\gamma}$ defined by \eqref{eq:zeroinflated}.
\end{proposition}
\begin{proof}
	Recalling $\xi \eid - \xi$, we observe that $\bfz_1,\ldots,\bfz_p$ are i.i.d. and uniformly at random drawn from the set
	$$
	\ml{K}_n=\{\pm \e_1, \ldots, \pm \e_n\}\subset \R^n.
	$$ Let
	$\bfT_n:=\bfZ_n^\top\bfZ_n=\sum_{i=1}^p \bfz_i \bfz_i^\top$. Then $\bfT_n=(T_{ij})$ is an $n\times n$ diagonal matrix, whose diagonal elements $(T_{11},\ldots, T_{nn})$ follow a multinomial distribution with probability mass function
	$$\P(T_{11}=t_1, \ldots, T_{nn}=t_n)= \frac{p!}{t_1! \cdots t_n!} \Big(\frac{1}{n}\Big)^p \qquad \text{ for } \quad t_1+\cdots+t_n=p\,.$$
	
	Thus, $F^{\bfT_n}(x):=n^{-1} \sum_{j=1}^n \1(\lambda_i(\bfT_n)\le x)=n^{-1} \sum_{j=1}^n \1(T_{jj}\le x)$ is a random discrete probability distribution with support $\N_0$, where
	$$
	F^{\bfT_n}(\{k\})=\frac{1}{n}\sum_{j=1}^n \1(T_{jj}=k),\qquad k\in \N_0.
	$$
	Then we have
	\begin{align*}
	\E[F^{\bfT_n}(\{k\})]=\P(T_{11}=k)=\binom{p}{k}\l(\frac{1}{n}\r)^k\l(1-\frac{1}{n}\r)^{p-k}\rw e^{-\gamma}\frac{\gamma^k}{k!}, \qquad \nto\,,
	\end{align*}
	and
	\begin{align*}
	\Var[F^{\bfT_n}(\{k\})]&=\frac{1}{n^2}\l[n\P(T_{11}=k)^2+n(n-1)\P(T_{11}=k,T_{22}=k)\r]-[\P(T_{11}=k)]^2\\
	&=\binom{p}{k}\binom{p-k}{k}\l(\frac{1}{n}\r)^{2k}\l(1-\frac{2}{n}\r)^{p-2k}-\l[\binom{p}{k}\l(\frac{1}{n}\r)^k\l(1-\frac{1}{n}\r)^{p-k}\r]^2+o(1)\\
	&\rw 0\,, \qquad \nto.
	\end{align*}
	Then by Chebyshev inequality, $F^{\bfT_n}(\{k\})\cip e^{-\gamma}\frac{\gamma^k}{k!}$ as $n\rw\infty$. That is to say, the LSD of $\bfT_n$ is $\operatorname{Poi}(\gamma)$. Since $\bfZ_n\bfZ_n^\top$ and $\bfZ_n^\top\bfZ_n$ have the same non-zero eigenvalues with the same multiplicities, we can use the relation
	$F^{\widetilde{\bfR}_n}=\frac{n}{p}F^{\bfT_n}+\l(1-\frac{n}{p}\r)\delta_0$
	to obtain for $\nto$
	$$
	F^{\widetilde{\bfR}_n}(\{0\})=\frac{n}{p}F^{\bfT_n}(\{0\})+\l(1-\frac{n}{p}\r)\cip\gamma^{-1}e^{-\gamma}+1-\gamma^{-1}
	$$
	and
	$$
	F^{\widetilde{\bfR}_n}(\{k\})=\frac{n}{p}F^{\bfT_n}(\{k\})\cip\gamma^{-1}e^{-\gamma}\frac{\gamma^k}{k!},
	$$
	Hence $F^{\widetilde{\bfR}_n}$ converges weakly in probability to $H_{0,\gamma}$.
\end{proof}
Now we are ready to derive the LSD of $\bfR_n$ in the slowly varying case.
\begin{theorem}\label{thm:R}
	If $\xi$ is slowly varying, then $F^{\bfR_n}$ converges weakly in probability to the zero-inflated Poisson distribution $H_{0,\gamma}$ defined by \eqref{eq:zeroinflated}.
\end{theorem}
\begin{proof}
	We notice from the definition of $\bfZ_n$ that $\diag(\bfR_n)=\diag(\widetilde{\bfR}_n)$ almost surely.
	By \cite[Corollary A.41]{bai:silverstein:2010}, the \Levy distance between the ESDs of $\bfR_n$ and $\widetilde{\bfR}_n$ converges to zero (in probability) if 
	\begin{equation}\label{eq:sdgd43}
		\frac{1}{p^2} \big(\tr(\bfR_n)+\tr(\widetilde{\bfR}_n) \big) \| \Y_n-\bfZ_n \|_F^2 \cip 0
	\end{equation}
	Using the definition of $\bfz_i$, we get
	\begin{align*}
		\frac{1}{p} \| \Y_n-\bfZ_n \|_F^2 &= \frac{1}{p} \sum_{i=1}^p \sum_{t=1}^n \big(Y_{it}-\sign(Y_{ik_i}) \1(t=k_i)\big)^2\\
		&=\frac{1}{p} \sum_{i=1}^p \left( 1- Y_{ik_i}^2 + (Y_{ik_i}-\sign(Y_{ik_i}))^2\right)\\
		&= \frac{2}{p}\Big( p-\sum_{i=1}^p |Y_{ik_i}|\Big).
	\end{align*}
	Therefore, we have 
	\begin{align*}
		\frac{1}{p} \E\big[\| \Y_n-\bfZ_n \|_F^2\big] &= 
		2-2 \E[|Y_{1k_1}|]\to 0\,, \qquad \nto,
	\end{align*}
	where for the last step we used $|Y_{1k_1}|\cip 1$ from the proof of Lemma~\ref{lem:structure}.
	Using Markov's inequality and the fact that $\tr(\bfR_n)=\tr(\widetilde{\bfR}_n)=p$, we conclude \eqref{eq:sdgd43}. Then Proposition~\ref{prop:Rtilde} finishes the proof.
\end{proof}

\section{Additional examples}\label{sec:additional_examples}

\jh{We give another deterministic example where the off-diagonal part contributes
and the limiting moments exist. 
\begin{example}
For fixed $\tau\in(0,1)$, set 
$m=\lfloor \tau n\rfloor$ and define
\[
	I_n=\{1,\ldots,m\},\qquad J_n=\{m+1,\ldots,n\}.
\]
Let $\A_n$ be the adjacency matrix of the complete bipartite graph with
vertex classes $I_n$ and $J_n$, that is,
\[
	a_{ij}^{(n)}
	=
	\begin{cases}
		1, &  i\in I_n,\ j\in J_n\ \text{or}\ i\in J_n,\ j\in I_n,\\
		0, & \text{otherwise}.
	\end{cases}
\]
Since $I_n$ and $J_n$ are disjoint, it follows $a_{ii}^{(n)}=0$ and
\[
	n^{-2}\|\A_n-\operatorname{\diag}(\A_n)\|_{\operatorname{F}}^2
	=
	\frac{2m(n-m)}{n^2}
	\rw 2\tau(1-\tau)>0,
\]
so Assumption \ref{ass:offdiag} fails. We now show that 
\[
	Q_{n,2}
	=
	2\left(\sum_{i\in I_n}Y_{1i}\right)
	 \left(\sum_{j\in J_n}Y_{1j}\right)
\]
 nevertheless has finite limiting moments. Since
the distribution of $\y_1$ is symmetric, all odd moments of $Q_{n,2}$ are zero.
For $\ell\in\N$,
\[
	\E[Q_{n,2}^{2\ell}]
	=
	2^{2\ell}
	\E\left[
		\left(\sum_{i\in I_n}Y_{1i}\right)^{2\ell}
		\left(\sum_{j\in J_n}Y_{1j}\right)^{2\ell}
	\right].
\]
Expanding both powers and using symmetry, only monomials in which every
coordinate appears with an even power contribute. Hence
\[
\begin{aligned}
	\frac{\E[Q_{n,2}^{2\ell}]}{2^{2\ell}}
	&=
	\sum_{r=1}^\ell\sum_{s=1}^\ell
	\frac{(2\ell)!^2}{r!s!}
	\sum_{\substack{k_1,\ldots,k_r\ge1\\ k_1+\cdots+k_r=\ell}}
	\sum_{\substack{h_1,\ldots,h_s\ge1\\ h_1+\cdots+h_s=\ell}}
	\frac{(m)_r(n-m)_s}
	{\prod_{a=1}^r(2k_a)!\prod_{b=1}^s(2h_b)!}  \beta_{2k_1,\ldots,2k_r,2h_1,\ldots,2h_s},
\end{aligned}
\]
where $(x)_r=x(x-1)\cdots(x-r+1)$. By Lemma~\ref{lem:allmoments},
$n^{r+s}\beta_{2k_1,\ldots,2k_r,2h_1,\ldots,2h_s}$
has a finite limit, while
\[
	\frac{(m)_r(n-m)_s}{n^{r+s}}
	\rw
	\tau^r(1-\tau)^s.
\]
Since the summations are finite for fixed $\ell$, the limit
\[
	q_{2\ell}:=\lim_{\nto}\E[Q_{n,2}^{2\ell}]
\]
exists and is finite. Moreover, the second moment is 
\[
	\E[Q_{n,2}^2]
	=
	4m(n-m)\beta_{2,2}
	\rw
	2\alpha \tau(1-\tau)>0,
\]
where we used $n^2\beta_{2,2}\to \alpha/2$. 

It remains to justify that the limiting even moments $(q_{2\ell})_{\ell\ge 1}$ determine a probability law.
We first derive upper bounds for $q_{2\ell}$. Since
\[
	(2k)!\ge 2^k k!,\qquad k\ge 1,
\]
we have
\[
	\prod_{a=1}^r(2k_a)!\prod_{b=1}^s(2h_b)!
	\ge
	2^{2\ell}\prod_{a=1}^r k_a!\prod_{b=1}^s h_b!,
\]
because \(k_1+\cdots+k_r=h_1+\cdots+h_s=\ell\). Hence, from the expression
for \(\E[Q_{n,2}^{2\ell}]\),
\[
\begin{aligned}
	\E[Q_{n,2}^{2\ell}]
	&\le
	(2\ell)!^2
	\sum_{r=1}^\ell\sum_{s=1}^\ell
	\binom{m}{r}\binom{n-m}{s}
	\sum_{\substack{k_1,\ldots,k_r\ge1\\ k_1+\cdots+k_r=\ell}}
	\sum_{\substack{h_1,\ldots,h_s\ge1\\ h_1+\cdots+h_s=\ell}}
	\frac{
		\beta_{2k_1,\ldots,2k_r,2h_1,\ldots,2h_s}
	}{
		\prod_{a=1}^r k_a!\prod_{b=1}^s h_b!
	}.
\end{aligned}
\]
The last sum has a simple interpretation. Indeed, expanding powers of the
squared coordinates gives
\[
\begin{aligned}
	&\sum_{r=1}^\ell\sum_{s=1}^\ell
	\binom{m}{r}\binom{n-m}{s}
	\sum_{\substack{k_1,\ldots,k_r\ge1\\ k_1+\cdots+k_r=\ell}}
	\sum_{\substack{h_1,\ldots,h_s\ge1\\ h_1+\cdots+h_s=\ell}}
	\frac{
		\beta_{2k_1,\ldots,2k_r,2h_1,\ldots,2h_s}
	}{
		\prod_{a=1}^r k_a!\prod_{b=1}^s h_b!
	}  \\
	&\hspace{2cm}
	=
	\frac{1}{(\ell!)^2}
	\E\left[
		\left(\sum_{i\in I_n}Y_{1i}^2\right)^\ell
		\left(\sum_{j\in J_n}Y_{1j}^2\right)^\ell
	\right].
\end{aligned}
\]
Since $\sum_{i\in I_n}Y_{1i}^2+\sum_{j\in J_n}Y_{1j}^2=1$,
we have
\[
	\left(\sum_{i\in I_n}Y_{1i}^2\right)
	\left(\sum_{j\in J_n}Y_{1j}^2\right)
	\le \frac14.
\]
Consequently,
\[
	\E[Q_{n,2}^{2\ell}]
	\le
	\frac{(2\ell)!^2}{(\ell!)^2}\left(\frac14\right)^\ell
	=
	\big((2\ell-1)!!\big)^2.
\]
Passing to the limit gives $q_{2\ell}
	\le
	\big((2\ell-1)!!\big)^2$. 
Therefore, by Stirling's formula,
\[
	q_{2\ell}^{-1/(2\ell)}
	\ge
	\big((2\ell-1)!!\big)^{-1/\ell}=\l(\frac{(2\ell)!}{2^\ell \ell!}\r)^{-1/\ell}
	\sim \frac{e}{2\ell}, \qquad  \ell \to \infty,
\]
and hence
$\sum_{\ell=1}^\infty q_{2\ell}^{-1/(2\ell)}=\infty$.
Thus the limiting moment sequence satisfies Carleman's condition.
Then $Q_{n,2}$ converges in distribution to the unique
moment-determinate law with these moments. Thus this deterministic complete
bipartite example violates the Frobenius-norm condition, has a non-degenerate
off-diagonal limit, and the limiting law is still characterized by its
moments.
\end{example}

\medskip

\begin{example}
We give another example where the off-diagonal part contributes and all
limiting moments exist. In this example the matrix $\mf{A}_n$ is taken to be the adjacency matrix of an Erdös--Rényi graph. Fix $\rho\in(0,1]$, and let
$(Z_{ij})_{1\le i<j\le n}$ be iid random variables with distribution
\[
	\P(Z_{ij}=1)=\P(Z_{ij}=-1)=\frac{\rho}{2},
	\qquad
	\P(Z_{ij}=0)=1-\rho .
\]
Assume that $(Z_{ij})_{1\le i<j\le n}$ is independent of $\y_1$, and define
\[
	a_{ii}^{(n)}=0,\qquad
	a_{ij}^{(n)}=a_{ji}^{(n)}=Z_{ij},\qquad 1\le i<j\le n.
\]
It is easy to calculate that
\[
	n^{-2}\E\|\A_n-\operatorname{\diag}(\A_n)\|_{\operatorname{F}}^2
	=
	n^{-2}\,2\sum_{1\le i<j\le n}\E Z_{ij}^2
	=
	\rho\,\frac{n(n-1)}{n^2}
	\rw \rho .
\]
Thus, Assumption \ref{ass:offdiag} fails whenever
$\rho>0$.

Since the distribution of $Z_{ij}$ is symmetric, all odd moments of
\[
	Q_{n,2}
	=
	2\sum_{1\le i<j\le n}Z_{ij}Y_{1i}Y_{1j} 
\]
are zero.
We now compute the even moments. Let
\[
	\mathcal E_n=\{(i,j):1\le i<j\le n\}.
\]
For a family $\mathbf h=(h_{ij})_{(i,j)\in\mathcal E_n}$ of nonnegative
integers, write
\[
	|\mathbf h|=\sum_{(i,j)\in\mathcal E_n}h_{ij}
	\qquad \text{ and } \qquad 
	d_i(\mathbf h)=\sum_{j=0}^{i-1} h_{ji}+\sum_{j=i+1}^n h_{ij}.
\]
By the multinomial theorem,
\[
\begin{aligned}
	Q_{n,2}^{2\ell}
	&=
	2^{2\ell}
	\sum_{\substack{(k_{ij})_{(i,j)\in\mathcal E_n}\\
	\sum_{i<j}k_{ij}=2\ell}}
	\frac{(2\ell)!}{\prod_{i<j}k_{ij}!}
	\prod_{i<j}\left(Z_{ij}Y_{1i}Y_{1j}\right)^{k_{ij}} .
\end{aligned}
\]
Taking expectation and using the independence and symmetry of the $Z_{ij}$'s,
all terms with some odd $k_{ij}$ vanish. Thus only terms of the form
$k_{ij}=2h_{ij}$ remain. Since
\[
	\E Z_{ij}^{2h}
	=
	\begin{cases}
		1, & h=0,\\
		\rho, & h\ge 1,
	\end{cases}
\]
we obtain
\[
\begin{aligned}
	\E[Q_{n,2}^{2\ell}]
	&=
	2^{2\ell}(2\ell)!
	\sum_{\substack{\mathbf h:\ |\mathbf h|=\ell}}
	\prod_{1\le i<j\le n}
	\frac{\rho^{\1_{\{h_{ij}>0\}}}}{(2h_{ij})!}
	\,
	\E\left[
		\prod_{i=1}^n Y_{1i}^{2d_i(\mathbf h)}
	\right].
\end{aligned}
\]
Here $h_{ij}$ records half of the multiplicity with which the edge $(i,j)$
appears in the expansion, and $d_i(\mathbf h)$ is the corresponding weighted
degree of vertex $i$.

Although the number of terms in the last sum grows with $n$, the sum has a
finite limit for every fixed $\ell$. To see this, group the terms according to
their weighted graph type. A weighted graph type $H$ is determined by the
finite graph formed by the edges with $h_{ij}>0$, together with their positive
weights. Since $|\mathbf h|=\ell$, each such graph has at most $\ell$ edges and at
most $2\ell$ vertices. Hence, for fixed $\ell$, only finitely many weighted graph
types can occur.

Fix one such type $H$. Suppose that $H$ has $r$ vertices, $s$ edges, edge
weights summing to $\ell$, and weighted vertex degrees
$d_1,\ldots,d_r$. The total contribution of all embeddings of this type into
$\{1,\ldots,n\}$ is of the form
\[
	C_H\,N_H(n)\,\beta_{2d_1,\ldots,2d_r},
\]
where $C_H$ depends only on $H$, $\ell$ and $\rho$, and where $N_H(n)$ denotes the
number of embeddings of $H$ into $\{1,\ldots,n\}$. In particular,
\[
	\frac{N_H(n)}{n^r}\rw c_H
\]
for some constant $c_H\ge0$. By Lemma~\ref{lem:allmoments},
$n^r\beta_{2d_1,\ldots,2d_r}$
has a finite limit and so does
\[
	N_H(n)\,\beta_{2d_1,\ldots,2d_r}\,.
\]
Since only finitely many weighted graph types occur for
fixed $\ell$, we conclude that
\[
	q_{2\ell}:=\lim_{\nto}\E[Q_{n,2}^{2\ell}] \in (0,\infty)\,.
\]
In particular,
\[
	\E[Q_{n,2}^2]
	=
	4\sum_{1\le i<j\le n}\E Z_{ij}^2\,\E[Y_{1i}^2Y_{1j}^2]
	=
	4\rho\binom n2\beta_{2,2}
	\rw
	\rho\alpha,
\]
where we used $n^2\beta_{2,2}\to\alpha/2$. 

It remains to show that the limiting moments determine a probability law. This
follows from a direct bound using the same expansion. Conditionally on $\y_1$,
we have
\[
\begin{aligned}
	\E[Q_{n,2}^{2\ell}\mid \y_1]
	&=
	2^{2\ell}(2\ell)!
	\sum_{\substack{\mathbf h:\ |\mathbf h|=\ell}}
	\prod_{1\le i<j\le n}
	\frac{\rho^{\mathbf 1_{\{h_{ij}>0\}}}
	(Y_{1i}^2Y_{1j}^2)^{h_{ij}}}{(2h_{ij})!}  \\
	&\le
	2^{2\ell}(2\ell)!
	\sum_{\substack{\mathbf h:\ |\mathbf h|=\ell}}
	\prod_{1\le i<j\le n}
	\frac{(Y_{1i}^2Y_{1j}^2)^{h_{ij}}}{(2h_{ij})!}.
\end{aligned}
\]
Using $(2k)!\ge 2^k k!$ and $|\mathbf h|=\ell$, we get
\[
	\prod_{1\le i<j\le n}(2h_{ij})!
	\ge
	2^\ell\prod_{1\le i<j\le n}h_{ij}!,
\]
and hence
\[
\begin{aligned}
	\E[Q_{n,2}^{2\ell}\mid \y_1]
	&\le
	2^\ell(2\ell)!
	\sum_{\substack{\mathbf h:\ |\mathbf h|=\ell}}
	\prod_{1\le i<j\le n}
	\frac{(Y_{1i}^2Y_{1j}^2)^{h_{ij}}}{h_{ij}!}  \le
	2^\ell\frac{(2\ell)!}{\ell!}
	\left(\sum_{1\le i<j\le n}Y_{1i}^2Y_{1j}^2\right)^\ell .
\end{aligned}
\]
Since $\sum_iY_{1i}^2=1$, we conclude by
taking expectation over $\y_1$ and passing to the limit that
\[
	q_{2\ell}\le 2^\ell\frac{(2\ell)!}{\ell!}
    \le (4\ell)^\ell.
\]
%We obtain
%\[
	%q_{2\ell}^{-1/(2\ell)} \ge (2\ell)^{-1/2}.
%\]
Consequently,
\[
	\sum_{\ell=1}^\infty q_{2\ell}^{-1/(2\ell)}=\infty,
\]
so Carleman's condition holds. Therefore, $Q_{n,2}$ converges in distribution to
the unique moment-determinate law with these moments.\end{example}
}

%%%%%%%%%%%%%%%%%%%%%%%%%%%%%%%%%%%%%%%%%%%%%%%%%%%%%%%%%%%
%%%%%%%%%%%%%%%%%%%%%%%%%%%%%%%%%%%%%%%%%%%%%%%%%%%%%%%%%%%
\section{Additional proofs}

\subsection{Verification of Example~\ref{ex:bounded}}\label{sec:verification-ex-bounded}
\begin{enumerate}
	\item[$(i)$] 
	To verify that the density integrates to one,
	let $\beta=\alpha/2$, with change of variable $x=t/(1+t)$ and $u=t^\beta$, we have
	\bea\label{eq-density-integral}
	\int_0^1 f_{\nu,\alpha}(x) dx&=\frac{\sin\l(\frac{\alpha\pi}{2}\r)}{\pi}\int_0^{+\infty} \frac{t^{\frac{\alpha}{2}-1}}{t^\alpha+2t^{\frac{\alpha}{2}}\cos\l(\frac{\alpha\pi}{2}\r)+1}dt\\
	&=\frac{\sin(\beta\pi)}{\beta\pi}\int_0^{+\infty}\frac{du}{u^2+1+2u\cos(\beta\pi)}=1.
	\eea
	Here in the last equality we use that for $\theta\in (0,\pi)$,
	\begin{align*}
		\int_0^{+\infty}\frac{du}{u^2+2u\cos\theta+1}&=\int_0^{+\infty}\frac{du}{(u+\cos\theta)^2+\sin^2\theta}\\
		&=\frac{1}{\sin\theta}\l[\arctan\l(\frac{u+\cos\theta}{\sin\theta}\r)\r]_0^{+\infty}=\frac{\theta}{\sin\theta}.
	\end{align*}
	
	To verify that $\mu_{\nu,\alpha}\cw\frac{1}{2}\delta_0+\frac{1}{2}\delta_1$ as $\alpha\downarrow0$ we proceed similarly to \eqref{eq-density-integral}. Setting $\theta=\beta\pi$ we get for (small) $\delta>0$ that
	\bea\nonumber
	\int_0^\delta f_{\nu,\alpha}(x) dx&=\frac{\sin\l(\frac{\alpha\pi}{2}\r)}{\pi}\int_0^{\delta/(1-\delta)} \frac{t^{\frac{\alpha}{2}-1}}{t^\alpha+2t^{\frac{\alpha}{2}}\cos\l(\frac{\alpha\pi}{2}\r)+1}dt\\&=\frac{\sin(\beta\pi)}{\beta\pi}\int_0^{\delta^\alpha/(1-\delta)^\alpha}\frac{du}{u^2+1+2u\cos(\beta\pi)}\\
	&\sim \frac{\sin(\beta\pi)}{\beta\pi}\int_0^{1}\frac{du}{u^2+1+2u\cos(\beta\pi)}\\
	&=\frac{1}{\theta}\l[\arctan\l(\frac{u+\cos\theta}{\sin\theta}\r)\r]_0^{1}=\frac{1}{2},    
	\eea
	where in the third line we let $\delta\to 0$ and in the last line we used
	$$
	\arctan\l(\frac{1+\cos\theta}{\sin\theta}\r)=\frac{\pi}{2}-\frac{\theta}{2}.
	$$
	Similarly,
	\bea\nonumber
	\int_{1-\delta}^1 f_{\nu,\alpha}(x) dx&=\frac{\sin(\beta\pi)}{\beta\pi}\int_{\delta^\alpha/(1-\delta)^\alpha}^{+\infty}\frac{du}{u^2+1+2u\cos(\beta\pi)}\\
	&\sim \frac{\sin(\beta\pi)}{\beta\pi}\int_1^{+\infty}\frac{du}{u^2+1+2u\cos(\beta\pi)}=\frac{1}{2}.
	\eea
	Therefore, it holds $\mu_{\nu,\alpha}\cw\frac{1}{2}\delta_0+\frac{1}{2}\delta_1$.
	
	To verify that $\mu_{\nu,\alpha}\cw\delta_{\frac{1}{2}}$ as $\alpha\uparrow2$,  let $X_\alpha$ have distribution $\mu_{\nu,\alpha}$ and consider 
	$S_\alpha=\beta\log\l(\frac{X_\alpha}{1-X_\alpha}\r)$.
	then the pdf of $S_\alpha$ is given by
	$$
	f_{S_\alpha}(s)=\frac{\sin\theta}{\theta}\cdot \frac{1}{e^s+e^{-s}+2\cos\theta},\quad s\in \R.
	$$
	Then for any $\eta>0$,
	$$
	\P(|S_\alpha|\geq\eta)\leq \frac{\sin\theta}{\theta}\int_{|s|\geq\eta}\frac{ds}{e^s+e^{-s}-2}\leq\frac{C_\eta\sin\theta}{\theta}\rw 0,\qquad \alpha\uparrow 2,
	$$
	since $\theta\uparrow\pi$. Thus, 
	$$
	S_\alpha\cip 0\implies X_\alpha\cip \frac{1}{2}\implies \mu_{\nu,\alpha}\cw \delta_{\frac{1}{2}}.
	$$
	\item[$(ii)$] The proof that $f_{\nu,\alpha}$ integrates to 1, and $\mu_{\nu,\alpha}\cw \delta_{\frac{1}{2}}$ as $\alpha\uparrow 2$, are similar to $(i)$ and therefore omitted. As $\alpha\downarrow0$, we have
	$$
	f_{\nu,\alpha}(x)\sim \frac{\sin(\alpha\pi/2)}{\alpha\pi/2}\cdot\frac{1}{x^2+(1-x)^2+2x(1-x)}\1_{(0,1)}(x)\sim \1_{(0,1)}(x).
	$$
	
\end{enumerate}

\subsection{Proof of Example~\ref{ex:unbounded}}

\begin{proof}[Proof of Example~\ref{ex:unbounded}]
	$(i)$\ For $\beta>-1,\ x>0$, we have
	$$
	\int_\R (x-y)_-^\beta \nu(dy)=\int_{0}^{+\infty}u^\beta e^{-(x+u)}du=e^{-x}\Gamma(1+\beta),
	$$
	\bea\nonumber
	\int_\R (x-y)_+^\beta \nu(dy)=e^{-x}\int_0^x u^\beta e^u du=e^{-x}\sum_{n=0}^\infty\frac{x^{\beta+n+1}}{n!(\beta+n+1)}=\frac{x^{\beta+1}e^{-x}}{\beta+1}\ {}_1F_1(\beta+1,\beta+2,x).
	\eea
	Here ${}_1F_1(a,b,x)$ is the Kummer's confluent hypergeometric function. By (13.2.23) in \cite{Book-Mathematical-Function}, except for $a=0,-1,-2,\cdots,$
	$$
	{_1F}_1(a,b,x)\sim\frac{\Gamma(b)}{\Gamma(a)}e^x x^{a-b},\qquad x\rw+\infty.
	$$
	Hence 
	$\int_\R (x-y)_+^\beta \nu(dy)\sim x^{\beta}$, as $x\to \infty$.
	By the density formula \eqref{Eq-Density-of-Quadratic-Form},
	\begin{align*}
		f(x)&\sim \frac{1}{\pi}\sin\l(\frac{\alpha\pi}{2}\r)\frac{x^{\frac{\alpha}{2}}e^{-x}\Gamma\l(\frac{\alpha}{2}\r)+x^{\frac{\alpha}{2}-1}e^{-x}\Gamma\l(1+\frac{\alpha}{2}\r)}{x^\alpha+e^{-2x}\Gamma\l(1+\frac{\alpha}{2}\r)^2+2\cos\l(\frac{\alpha\pi}{2}\r)x^{\frac{\alpha}{2}}e^{-x}\Gamma\l(1+\frac{\alpha}{2}\r)}\\
		&\sim \frac{1}{\pi}\sin\l(\frac{\alpha\pi}{2}\r)\Gamma\l(\frac{\alpha}{2}\r)x^{-\frac{\alpha}{2}}e^{-x}\\
		&= \frac{1}{\Gamma\l(1-\frac{\alpha}{2}\r)}x^{-\frac{\alpha}{2}}e^{-x}.
	\end{align*}
	$(ii)$\ For $\beta>-1,\ x>0,$ with a change of variable $y=\frac{x}{1-v}$,
	$$
	\int_\R(x-y)_-^\beta\nu(dy)=sx^{\beta-s}\int_0^1v^\beta(1-v)^{-\beta+s-1}dv=sx^{\beta-s}B(\beta+1,s-\beta).
	$$
	With a change of variable $y=x(1-t)$, by (8.17.7) in \cite{Book-Mathematical-Function},
	\bea\nonumber
	\int_\R(x-y)_+^\beta\nu(dy)&=sx^{\beta-s}\int_0^{1-x^{-1}}t^\beta(1-t)^{-s-1}dt\\&=\frac{sx^{\beta-s}}{\beta+1}(1-x^{-1})^{\beta+1}{}_2F_1(\beta+1,s+1;\beta+2;1-x^{-1}).
	\eea
	Here ${}_2F_1(a,b;c,x)$ is the Gaussian hypergeometric function. By (15.4.23) in \cite{Book-Mathematical-Function}, when $c-a-b<0$,
	$$
	\lim _{z \uparrow 1} \frac{{}_2F_1(a, b ; c ; z)}{(1-z)^{c-a-b}}=\frac{\Gamma(c) \Gamma(a+b-c)}{\Gamma(a) \Gamma(b)} .
	$$
	Then as $x\rw\infty$,
	$$
	\int_\R(x-y)_+^\beta\nu(dy)\sim x^{\beta}(1-x^{-1})^{\beta+1}\sim x^\beta.
	$$
	Then by the density formula we have \eqref{Eq-Quadratic-form-nu-pareto-tail-probability}.
\end{proof}

\subsection{Proof of Proposition~\ref{Prop-concentration-resolvant-diagonal}}

\begin{proof}[Proof of Proposition~\ref{Prop-concentration-resolvant-diagonal}]
	Let $\B_{n,k}(z)=(\Y_n^\top \Y_n-\y_k \y_k^\tp-z\bfI_n)^{-1}$ and set 
	$$L_n(f,z)=\frac{1}{n} \sum_{j=1}^n f\big( \B_n(z)_{jj} \big),\quad L_{n,k}(f,z)=\frac{1}{n}\sum_{j=1}^nf(\B_{n,k}(z)_{jj}).$$
	Let $\ml{F}_j$ be the $\sigma$-algebra generated by $\{\y_i:1\leq i\leq j\}$, and $\ml{F}_0=\{\emptyset,\Omega\}$ the trivial $\sigma$-algebra. Let $M_k=\E[L_n(f,z)|\ml{F}_k]$, then $\{M_k\}_{k=0}^p$ is a martingale with $M_0=\E[L_n(f,z)],\ M_p=L_n(f,z)$. Since $L_{n,k}(f,z)$ is independent of $\y_k$, we have
	$$
	\E[L_{n,k}(f,z)|\ml{F}_k]=\E[L_{n,k}(f,z)|\ml{F}_{k-1}].
	$$
	Therefore,
	$$
	|M_k-M_{k-1}|\leq |\E[L_n(f,z)-L_{n,k}(f,z)|\ml{F}_k]|+|\E[L_n(f,z)-L_{n,k}(f,z)|\ml{F}_{k-1}]|.
	$$
	and using the Lipschitz property of $f$ we get
	\begin{align*}
		|L_n(f,z)-L_{n,k}(f,z)|&\leq \frac{1}{n}\sum_{j=1}^n |f\big(\B_n(z)_{jj}\big)-f\big(\B_{n,k}(z)_{jj}\big)|\\
		&\leq \frac{L_f}{n}\sum_{j=1}^n |\B_n(z)_{jj}-\B_{n,k}(z)_{jj}|.
	\end{align*}
	{By virtue of the Sherman-Morrison formula, we have
		\bea\nonumber
		\B_n(z)-\B_{n,k}(z)=-\frac{\B_{n,k}(z)\y_k\y_k^\top \B_{n,k}(z)}{1+\y_k^\top \B_{n,k}(z)\y_k},
		\eea
		and consequently
		$$
		\B_n(z)_{jj}-\B_{n,k}(z)_{jj}=-\frac{(\B_{n,k}(z)\y_k)_j^2}{1+\y_k^\top \B_{n,k}(z)\y_k}.
		$$
		Then
		\bea\nonumber
		\sum_{j=1}^n|\B_n(z)_{jj}-\B_{n,k}(z)_{jj}|= \frac{\y_k^\tp \B_{n,k}(z)^*\B_{n,k}(z)\y_k}{|1+\y_k^\top \B_{n,k}(z)\y_k|}.
		\eea
		If $\Im z>0$, by 
		$$
		|1+\y_k^\top \B_{n,k}(z)\y_k|\geq \Im (\y_k^\top \B_{n,k}(z)\y_k)=\y_k^\top \Im(\B_{n,k}(z))\y_k=\y_k^\top ((\Im z)\B_{n,k}(z)^*\B_{n,k}(z))\y_k,
		$$
		we have
		$$
		\sum_{j=1}^n|\B_n(z)_{jj}-\B_{n,k}(z)_{jj}|\leq\frac{1}{\Im z}.
		$$
		If $\Im z=0,\ \Re z<0$, then $\B_{n,k}(z)\in \R^{n\times n}$. Write
		$
		\B_{n,k}(z)=\sum_{i=1}^n \frac{\mf{v}_i\mf{v}_i^{\tp}}{\lambda_i-z},
		$ where $\lambda_1\geq \cdots\geq \lambda_n\geq 0$ are eigenvalues of $\Y_n^\tp\Y_n-\y_k^\tp \y_k$.
		Then
		\bea\nonumber
		\sum_{j=1}^n|\B_n(z)_{jj}-\B_{n,k}(z)_{jj}|&=\frac{\y_k^\tp \B_{n,k}(z)^2 \y_k}{1+ \y_k^\tp \B_{n,k}(z) \y_k}=\frac{\sum_{i}\frac{(\y_k^\tp \mf{v}_i)^2}{(\lambda_i-z)^2}}{1+\sum_i\frac{(\y_k^\tp \mf{v}_i)^2}{\lambda_i-z}}\leq \frac{\frac{1}{\lambda_n-z}\sum_{i}\frac{(\y_k^\tp \mf{v}_i)^2}{\lambda_i-z}}{1+\sum_i\frac{(\y_k^\tp \mf{v}_i)^2}{\lambda_i-z}}\leq \frac{1}{|\Re z|}.
		\eea
		%Then from $|\y_k|=1$,
		%$$
		%|L_n(f,z)-L_{n,k}(f,z)|\leq \frac{L_f}{n}\sum_{j=1}^n \frac{\y_k\e_j\e_j^\top \y_k^\top}{\operatorname{dist}(z,\R^+)^2}=\frac{L_f}{n}\cdot\frac{1}{\operatorname{dist}(z,\R^+)^2}.
		%$$
		Then for $z\in \C\backslash[0,+\infty)$,
		$$
		\sum_{j=1}^n|\B_n(z)_{jj}-\B_{n,k}(z)_{jj}|\leq\frac{1}{\operatorname{dist}(z,\R^+)}.
		$$
		This implies
		$$
		|M_k-M_{k-1}|\leq \frac{2L_f}{n}\cdot\frac{1}{\operatorname{dist}(z,\R^+)}.
		$$
	}
	and by Azuma's inequality we finish the proof.
\end{proof}

%%%%%%%%%%%%%%%%%%%%%%%%%%%%%%%%%%%%%%%%%%%%%%
%% Multiple Appendixes:                     %%
%%%%%%%%%%%%%%%%%%%%%%%%%%%%%%%%%%%%%%%%%%%%%%
\begin{appendix}
\section{A Hanson--Wright inequality for vectors on the unit sphere}
Consider an $n$-dimensional random vector  
$\bfx = (X_1, \ldots, X_n)^{\top} \in \R^n$,
whose components $X_i$ are iid copies of some non-degenerate random variable $\xi$. For a matrix $\A\in \R^{n\times n}$, the quadratic form of the self-normalized random vector $\y=\x/\twonorm{\x}$ is defined as $Q_n(\y,\A)=\y^\top \A \y$.

Recall that the subgaussian norm of  a real-valued random variable $Z$ is defined as
\begin{align*}
	\norm{Z}_{\psi_2}:=\inf\{\nu>0 : \E [\psi_2(|Z|/\nu)]\leq 1\}\,,
\end{align*}
where $\psi_2(x)=\exp(x^2)-1$.

\begin{theorem}[Hanson--Wright inequality for $Q_n(\y,\A)$]\label{thm:HW-Qy}
	If $\xi$ is centered and sub-Gaussian, then there exist universal constants $C,c>0$ such that for any $n\in \mathbb{N}$ and $t>0$,
	\begin{equation}\label{eq:hw}
		\P\l(\Big|Q_n(\y,\A)-\E[Q_n(\y,\A)]\Big|\geq t\r)\leq C\exp\l(-c\min\l(\frac{nt^2}{K^4\|\A\|^2},\ \frac{n t}{K^2\|\A\|}\r)\r),
	\end{equation}
	where $K=\|\xi\|_{\psi_2}/\sqrt{\Var(\xi)}$.
\end{theorem}

\begin{proof}
	For simplicity we write $Q_\x=\x^\top \A \x$ and $Q_\y=Q_n(\y,\A)$.
	In the following proof the constants $c,C$ might change form line to line. Since we can multiply $\x$ with any positive constant without changing the value of $Q_\y$, we may assume $\E[\xi^2]=1$ without loss of generality.
	Note that $\E[\|\x\|_2^2]=n$ and 
	$$\E\big[Q_\y\big]=\frac{\tr(\A)}{n} {+ \theta_n \bigg( -\tr(\A) +\sum_{i,j=1}^n \A_{ij} \bigg)}=\frac{\E[Q_\x]}{\E[\|\x\|_2^2]}{+o(\|\A\|/n)}$$ 
	where we used $\theta_n:=\E\big[ X_1X_2/ \|x\|_2^2 \big]=o(n^{-2})$ by \cite{gine:goetze:mason:1997}, $|\tr(\A)|\le n \|\A\|$ and $\big|\sum_{i,j=1}^n \A_{ij}\big| \le n \|\A\|$. An inspection of the \rhs\ in \eqref{eq:hw} shows that the $o(\|\A\|/n)$-term is negligible; we will assume it to be zero in the following. Then we have
	\begin{align*}
		Q_\y -\E[Q_\y]
		&=\frac{Q_\x}{\|\x\|_2^2}-\frac{\E[Q_\x]}{n}\\
		&=\frac{Q_\x}{\|\x\|_2^2n}\Big(\|\x\|_2^2-n\Big)+\frac{Q_\x-\E[Q_\x]}{n}.
	\end{align*}
	Using Theorem \ref{thm:classical-HW} and 
	$$
	\frac{|Q_\x|}{\|\x\|_2^2\E[|\|\x\|_2^2|]}\leq \frac{\|\A\|}{n},
	$$
	we obtain
	$$
	\P\l(\frac{|Q_\x|}{\|\x\|_2^2n}\Big|\|\x\|_2^2-n\Big|\geq t\r)
	%\leq \P\l(\Big|\|\x\|_2^2-n\Big|\geq \frac{nt}{\|\A\|}\r)
	\leq C\exp\l(-c\min\l(\frac{nt^2}{K^4\|\A\|^2},\ \frac{nt}{K^2\|\A\|}\r)\r),
	$$
	$$
	\P\l(\Big|Q_\x-\E[Q_\x]\Big|\geq nt\r)\leq C\exp\l(-c\min\l(\frac{n^2t^2}{K^4\|\A\|_\text{F}^2},\ \frac{nt}{K^2\|\A\|}\r)\r),
	$$
	since $\|\I_n\|_{\text{F}}=\sqrt{n}$. Then noting that $\|\A\|_{\text{F}}\leq \sqrt{n} \|\A\|$, we have
	$$
	\P\l(\Big|Q_{\y}-\E[Q_{\y}]\Big|\geq t\r)\leq C\exp\l(-c\min\l(\frac{nt^2}{K^4\|\A\|^2},\ \frac{nt}{K^2\|\A\|}\r)\r).
	$$
\end{proof}

\section{Auxiliary results}

The following lemma is an extension of Proposition 8 in \cite[Section 3.1]{mingo2017free}.
\begin{lemma}\label{Lemma-Recover-atom}
	Let $\beta\in \R^+$ and $\tau$ be a probability measure on $\R$. Then for any $a\in \R$ it holds
	\begin{equation*}
	\lim_{\sphericalangle z\rw a} (z-a)^\beta\int_\R (z-x)^{-\beta}\tau(dx)=\tau(\{a\}).
	\end{equation*}
\end{lemma}
\begin{proof}
	Given $\theta>0$ and $z=u+iv$ with $|u-a|<\theta v$, then we have for any $x\in \R$,
	$$
	\left | \frac{z-a}{z-x}\r |^2=\frac{(u-a)^2+v^2}{(u-x)^2+v^2}=\frac{1+\l(\frac{u-a}{v}\r)^2}{1+\l(\frac{u-x}{v}\r)^2}\leq 1+\l(\frac{u-a}{v}\r)^2<1+\theta^2.
	$$
	Letting $\sigma=\tau-\tau(\{a\})\delta_a$ we get
	$$
	(z-a)^\beta\int_\R (z-x)^{-\beta}\tau(dx)=\int_\R\l(\frac{z-a}{z-x}\r)^\beta\sigma(dx)+\tau(\{a\}).
	$$
	The dominated convergence theorem implies, as $\sphericalangle z \rightarrow a$,
	$$
	\l|(z-a)^\beta\int_\R (z-x)^{-\beta}\tau(dx)-\tau(\{a\})\r|\leq \int_\R\l|\frac{z-a}{z-x}\r|^\beta\sigma(dx)\rw 0\,,
	$$
	which finishes the proof of the lemma.
\end{proof}

\begin{lemma}\label{Lemma-a.e.finite-alphaPotential}
	Let $\tau$ be a probability measure on $\R$ and $\beta\in (0,1)$. For $x\in\R$, define
	$$
	I_{\tau,\beta}(x)=\int_\R \frac{\tau(dy)}{|x-y|^\beta},
	$$
	then $I_{\tau,\beta}(x)$ is finite $\lambda$-a.e. on $\R$.
\end{lemma}
\begin{proof}
	For fixed $\delta\in \R^+$ we have 
	$$
	I_{\tau,\beta}(x)=\int_{|x-y|<\delta}\frac{\tau(dy)}{|x-y|^\beta}+\int_{|x-y|\geq\delta}\frac{\tau(dy)}{|x-y|^\beta}:=I_1(x)+I_2(x).
	$$
	Then $I_2(x)\leq \delta^{-\beta}$ and an application of Fubini-Tonelli theorem yields
	\begin{align*}
		\int_\R I_1(x)dx&=\int_\R\l(\int_\R \frac{\1(|x-y|<\delta)}{|x-y|^\beta}\tau(dy)\r)dx\\
		&=\int_\R\l(\int_\R \frac{\1(|x-y|<\delta)}{|x-y|^\beta}dx\r)\tau(dy)\\
		&=\int_\R\l(\int_{-\delta}^\delta \frac{1}{|t|^\beta}dt\r)\tau(dy)\\
		&=\frac{2\delta^{1-\beta}}{1-\beta}<\infty.
	\end{align*}
	Therefore, $I_1(x)$ (and hence $I_{\tau,\beta}(x)$) is finite $\lambda$-a.e. on $\R$. 
\end{proof}

The following lemma is a direct corollary of Theorem 3.2 in \cite{BillingsleyBookConvergence}.
\begin{lemma}\label{lemma-Billingsley-Thm2.3}
	Let $Q_{n,m}$, $T_n$ ($n,m\in \N$) be random variables. Suppose that:
	\begin{enumerate}
		\item For fixed $m$, as $n\rw\infty$, $Q_{n,m}\cid\bar{Q}_m$,\item As $m\rw\infty$, $\bar{Q}_{m}\cid Q$, 
		\item $\lim_{m\rw\infty}\limsup_{n\rw\infty}\E[|T_n-Q_{n,m}|]=0.$
	\end{enumerate}
	Then $T_n\cid Q$.
\end{lemma}
\end{appendix}

%%%%%%%%%%%%%%%%%%%%%%%%%%%%%%%%%%%%%%%%%%%%%%
%% Support information, if any,             %%
%% should be provided in the                %%
%% Acknowledgements section.                %%
%%%%%%%%%%%%%%%%%%%%%%%%%%%%%%%%%%%%%%%%%%%%%%
\begin{acks}[Acknowledgments]
The authors would like to thank Profs. Dang-Zheng Liu and Dong Wang for insightful discussions on the materials in Section~\ref{sec:app-RMT}, and  Prof. Cosme Louart for his contribution to Theorem \ref{thm:HW-Qy} in the Appendix. \jh{We are grateful for the constructive feedback by the editor and associate editor which led to an improved version of this paper.} 
\end{acks}
%%%%%%%%%%%%%%%%%%%%%%%%%%%%%%%%%%%%%%%%%%%%%%
%% Funding information, if any,             %%
%% should be provided in the                %%
%% funding section.                         %%
%%%%%%%%%%%%%%%%%%%%%%%%%%%%%%%%%%%%%%%%%%%%%%
\begin{funding}
J.\ Heiny’s research was supported by the Swedish Research Council grant VR-2023-03577 ``High-dimensional extremes and random matrix structures'' and by the Verg-Foundation.
\end{funding}

%%%%%%%%%%%%%%%%%%%%%%%%%%%%%%%%%%%%%%%%%%%%%%
%% Supplementary Material, including data   %%
%% sets and code, should be provided in     %%
%% {supplement} environment with title      %%
%% and short description. It cannot be      %%
%% available exclusively as external link.  %%
%% All Supplementary Material must be       %%
%% available to the reader on Project       %%
%% Euclid with the published article.       %%
%%%%%%%%%%%%%%%%%%%%%%%%%%%%%%%%%%%%%%%%%%%%%%
%\begin{supplement}
%\stitle{???}
%\sdescription{???.}
%\end{supplement}

%%  MR numbers will be added by VTeX.                      %%
%%                                                         %%
%%  Use \cite{...} to cite references in text.             %%
%%                                                         %%
%%%%%%%%%%%%%%%%%%%%%%%%%%%%%%%%%%%%%%%%%%%%%%%%%%%%%%%%%%%%%

%% if your bibliography is in bibtex format, uncomment commands:
\bibliographystyle{imsart-number} % Style BST file (imsart-number.bst or imsart-nameyear.bst)
\bibliography{reference2026}       % Bibliography file (usually '*.bib')

%% or include bibliography directly:
% \begin{thebibliography}{}
% \bibitem{b1}
% \end{thebibliography}

\end{document}